\documentclass[12pt,notitlepage,reqno]{amsart}
\usepackage{setspace}
\usepackage{sty}

\title[Low lying zeros of Rankin-Selberg $L$-functions]{Low lying zeros of Rankin-Selberg $L$-functions}

\author{Alexander Shashkov}
\address{Department of Mathematics, University of California, Berkeley}
\email{ashashkov@berkeley.edu}


\thanks{Declarations of interest: none}

\date{\today}

\begin{document}

\begin{abstract}
    We study the low lying zeros of $GL(2) \times GL(2)$ Rankin-Selberg $L$-functions. Assuming the Generalized Riemann Hypothesis, we compute the $1$-level density of the low-lying zeroes of $L(s, f \otimes g)$ averaged over families of Rankin-Selberg convolutions, where $f, g$ are cuspidal newforms with even weights $k_1, k_2$ and prime levels $N_1, N_2$, respectively. The Katz-Sarnak density conjecture predicts that in the limit, the $1$-level density of suitable families of $L$-functions is the same as the distribution of eigenvalues of corresponding families of random matrices. The 1-level density relies on a smooth test function $\phi$ whose Fourier transform $\widehat\phi$ has compact support. In general, we show the Katz-Sarnak density conjecture holds for test functions $\phi$ with $\operatorname{supp} \widehat\phi \subset (-\frac{1}{2}, \frac{1}{2})$. When $N_1 = N_2$, we prove the density conjecture for $\operatorname{supp} \widehat\phi \subset (-\frac{5}{4}, \frac{5}{4})$ when $k_1 \ne k_2$, and $\operatorname{supp} \widehat\phi \subset (-\frac{29}{28}, \frac{29}{28})$ when $k_1 = k_2$. A secondary term contributes to the 1-level density when the support of $\widehat\phi$ exceeds $(-1, 1)$, which makes these results particularly interesting. The main idea which allows us to extend the support of $\widehat\phi$ beyond $(-1, 1)$ is an analysis of the products of Kloosterman sums arising from the Petersson formula. We also carefully treat the contributions from poles in the case where $k_1 = k_2$. Our work provides conditional lower bounds for the proportion of Rankin-Selberg $L$-functions which are non-vanishing at the central point and for a related conjecture of Keating and Snaith on central $L$-values.
\end{abstract}

\maketitle

\setcounter{tocdepth}{1}
\tableofcontents

\section{Introduction}\label{sec:intro}
Since Montgomery and Dyson's discovery that the two point correlation of the zeros of the Riemann zeta function agrees with the pair correlation function for eigenvalues of the Gaussian Unitary Ensemble (see \cite{Mon}), the connection between the zeros of $L$-functions and the eigenvalues of random matrices has been a major area of study. It is now widely believed that the statistical behavior of families of $L$-functions can be modeled by ensembles of random matrices. Based on the observation that the spacing statistics of high zeros associated with cuspidal $L$-functions agree with the corresponding statistics for eigenvalues of random unitary matrices under Haar measure (see \cite{RS}, for example), it was originally believed that only the unitary ensemble was important to number theory. However, Katz and Sarnak \cite{KS1, KS2} showed that these statistics are the same for all classical compact groups.
These statistics, the $n$-level correlations, are unaffected by finite numbers of zeros. In particular, they fail to identify differences in behavior near $s = 1/2$. 

The $n$-level density statistic was introduced to distinguish the behavior of families of $L$-functions close to the central point. Based partially on an analogy with the function field setting, Katz and Sarnak conjectured that the low-lying zeros of families of $L$-functions behave like the eigenvalues near $1$ of classical compact groups (unitary, symplectic, and orthogonal). The behavior of the eigenvalues near 1 is different for each matrix group. A growing body of evidence has shown that this conjecture holds for test functions with suitably restricted support for a wide range of families of $L$-functions. For a non-exhaustive list, see \cite{alpoge2015low, Alpoge2015, barrett2017one, cohen2022extending, devin2022low, drappeau2023one, duenez2006low, DM, Entin2012, FiMi, Gao, Guloglu2005, HM, ILS, knightly2019weighted, Mil2, MP, OS1, OS2, Ricotta2007, Ro, Ru, sarnak2016families, Shin2012, waxman2021lower, yang2009,  Yo}.

We study the 1-level density of families of Rankin-Selberg $L$-functions, which are the $L$-functions associated with Rankin-Selberg convolutions of cusp forms. In particular, let $H^*_k(N)$ denote the set of cusp forms of weight $k$ which are newforms of prime level $N$ (see the next section for more detail). We assume that the level $N$ is prime in order to make computations easier, but our results should hold for any $N$; see \cite{barrett2017one}.

Let $\phi$ be an even smooth test function whose Fourier transform has compact support. Take $f \in H_{k_1}^*(N_1)$, $g \in H_{k_2}^*(N_2)$, and let $L(s, f\ot g)$ be the Rankin-Selberg convolution $L$-function. See Section \ref{sec: RSL} for a precise definition. By the work of Rankin \cite{rankin1939contributions}, Selberg \cite{selberg1940bemerkungen}, and Moeglin and Waldspurger \cite{moeglin1989poles}, $L(s, f \ot g)$ is holomorphic in the entire complex plane except for a simple pole at $s = 1$ when $f = \overline{g}$. Here $\overline{g}$ is the dual of $g$, that is, the cusp form with $n$th Fourier coefficient $\overline{a_g(n)}$, where $a_g(n)$ is the $n$th Fourier coefficient of $g$. Note that since our forms have trivial central character, our cusp forms are self-dual so we have that $f = \overline{f}$ so that $L(s, f\ot g)$ has a pole if and only if $f = g$.  We are interested in the quantity
\begin{equation}\label{eq: density def}
    D(f \ot g; \phi) := \sum_{\rho_{f \ot g}} \phi\left( \frac{\gamma_{f \ot g}}{2\pi} \log R \right)
\end{equation}
where the sum is over the non-trivial zeros $\rho_{f\ot g} = \frac{1}{2} + i \gamma_{f \ot g}$ of $L(s, f\ot g)$ and $R$ is the analytic conductor of $f \ot g$. We have that
\begin{equation}\label{eq: conductor conductor}
    R = \begin{cases} [N_1, N_2]^2 (k_1 - k_2)^2 (k_1 +k_2)^2 & k_1 \ne k_2 \\
    [N_1, N_2]^2 k_1^2 & k_1 = k_2.
    \end{cases}
\end{equation}
See Section \ref{sec: RSL} for more information on the conductor. Because much of our analysis relies on bounding sums over primes smaller than $R^\sigma$ where $\supp \hphi \subset (-\sigma, \sigma)$, we are able to obtain better results when the conductor is small. These small conductor cases allow us to obtain Fourier support up to and beyond $(-1, 1)$ below, while we are restricted to $(-1/2, 1/2)$ in general.

For the purposes of this paper, we assume the Generalized Riemann Hypothesis (GRH) for $L(s, f \ot g)$, so that $\gamma_{f \ot g}$ is always real. We assume GRH for $L(s, f \ot g)$ as well as $L(s, \sym^2 (f))$, $L(s, \sym^2 (g))$ and $L(s, \sym^2 (f) \ot \sym^2 (g))$ in order to obtain better estimates on prime sums in Section \ref{sec:explicit}, but this also makes our results easier to interpret. As such, all of the results stated in this paper are dependent on GRH for these $L$-functions. In order to prove Theorem \ref{thm: extendo}, we also assume GRH for Dirichlet $L$-functions.

We are interested in averages of $D(f \ot g; \phi)$ over families of Rankin-Selberg convolutions of cusp forms. In particular, let 
\begin{equation}
    H(k_1, N_1, k_2, N_2) = \{f \ot g \ | \ f \in H_{k_1}^*(N_1), g \in H_{k_2}^*(N_2)\}
\end{equation}
be the family of Rankin-Selberg convolutions of cusp forms from $H_{k_1}^*(N_1)$ and $H_{k_2}^*(N_2)$. These are $GL(4)$ automorphic forms, which are difficult to study in general. However, by studying Rankin-Selberg convolutions, we are able to apply the $GL(2)$ Petersson trace formula in order to make our calculations tractable. As such, our paper mostly follows the method of \cite{ILS}, where the 1-level density was studied for families of cusp forms. We utilize results from \cite{ILS} wherever possible for brevity. The main novelty in our method comes from studying the interaction between the terms arising from applying the Petersson formula, and from our analysis of the contribution from the poles (which did not appear in \cite{ILS}).

The families of forms we study exhibit symplectic symmetry, as shown by Due\~{n}ez and Miller \cite{DM}. Due\~{n}ez and Miller study convolutions of families of $L$-functions in general, and are able to determine the symmetry type for a large variety of families. However, their results do not give explicit bounds on the support, and their methods are not strong enough in order to obtain the support proved in this paper. In particular, assuming GRH and using the estimates in \cite{ILS} to obtain explicit results for our family of study, the maximum support obtainable using the methods of \cite{DM} is $(-1/5, 1/5)$ in general, with an extension to $(-2/5, 2/5)$ in certain cases. See Remark~\ref{rem: DM limit} for details. Notably, this is not strong enough to show that a positive proportion of $L$-functions in the family vanish at the central point, and weaker than the results of this paper. Due\~{n}ez and Miller \cite{duenez2006low} also study convolutions of families of cusp forms by a single, fixed, Hecke-Maass cusp form, but do not obtain an explicit bound on the support of the test function. Shin and Templier \cite{Shin2012} study 1-level densities for very general families of automorphic forms, but also do not obtain explicit support. 

The density function for the symplectic group equals
\begin{equation}\label{eq: W def}
    W(Sp)(x) \clq 1 - \frac{\sin 2\pi x}{2\pi x}.
\end{equation}
The Katz-Sarnak density conjecture predicts that in the limit, as $D(f \ot g, \phi)$ is averaged over an increasingly large family, the 1-level density equals
\begin{equation}
    \infint \phi(x) W(Sp(x)) dx = \infint \hphi(y) \widehat W(Sp(y)) dy.
\end{equation}
The Fourier transform of the symplectic density function is
\begin{equation}\label{eq: W hat}
    \widehat{W}(Sp)(y) = \delta(y) - \frac{1}{2} \eta(y),
\end{equation}
where $\delta$ is the Dirac delta function and
\begin{equation}
    \eta(y) \clq \begin{cases}
        1 & |y| < 1 \\
        \frac{1}{2} & |y| = 1 \\
        0 & |y| > 1.
    \end{cases}
\end{equation}
It follows that if $ \supp \hphi \subset (-1, 1)$ then
\begin{equation}
    \infint \phi(x) W(Sp(x)) dx = \hphi(0) - \frac{1}{2}\phi(0).
\end{equation}
Because of the discontinuity of $\widehat{W}(Sp)(y)$ at $\pm 1$, results which allow us to take the support of $\hphi$ beyond the interval $(-1, 1)$ are particularly interesting. We prove such a result in Theorem \ref{thm: extendo}.

We first state a general result with more restricted support. 
\begin{thm}\label{thm: general result}
Assume GRH. Fix a test function $\phi$ with $\supp \hphi \subset (-1/2, 1/2)$, let $k_1, k_2$ be even integers and let $N_1, N_2$ be primes. We have
\begin{equation}
    \lim_{N_1N_2 \to \infty} \frac{1}{|H(k_1, N_1, k_2, N_2)|}\sum_{f \ot g \in H(k_1, N_1, k_2, N_2)}  D(f \ot g; \phi) = \infint \phi(x) W(Sp(x)) dx.
\end{equation}
\end{thm}
Note that we just need $N_1N_2 \to \infty$, so that we can hold $N_1$ or $N_2$ fixed and let the other grow, or allow them both to grow in unison. 

When $N_1 = N_2$, we are able to extend the support because the analytic conductor of our cusp forms is smaller; see \eqref{eq: conductor conductor}.

\begin{thm}\label{thm: extendo}
    Assume GRH and let $k_1, k_2$ be even integers. If $k_1 \ne k_2$, fix a test function $\phi$ with $\supp \hphi \subset(-5/4, 5/4)$. If $k_1 = k_2$, then take $\supp \hphi \subset(-29/28, 29/28)$. Then as $N \to \infty$ through the primes, we have
\begin{equation}
    \lim_{N \to \infty} \frac{1}{|H(k_1, N, k_2, N)|}\sum_{f \ot g \in H(k_1, N, k_2, N)} D(f \ot g; \phi) = \infint \phi(x) W(Sp(x)) dx.
\end{equation}
\end{thm}

\begin{rem}
    When $k_1 = k_2$, for each $f \in H_{k_1}^*(N)$, the $L$-function $L(s, f\ot f)$ is in our family and has a pole at $s = 1$. This pole appears in the explicit formula, and contributes to the 1-level density if the support of $\hphi$ exceeds $(-1, 1)$. See Remark \ref{rem: pole fail} for more details.
    
    In general, this phenomenon makes it difficult to obtain 1-level density results with support exceeding $(-1, 1)$ when the family contains non-entire $L$-functions, and there are only a limited number of results of this type. Fouvry and Iwaniec \cite{FI} study Hecke $L$-functions (some of which have poles) and are able to obtain support $(-4/3, 4/3)$ by utilizing GRH and additional averaging.
\end{rem}

If we take the weight of our forms to infinity, we can prove a similar result. 
\begin{thm}\label{thm: over weight}
Assume GRH. Fix a test function $\phi$ with $\supp \hphi \subset (-1/2, 1/2)$ and let $N_1, N_2$ be 1 or prime. Then we have
\begin{equation}
    \lim_{k_1 k_2 \to \infty} \frac{1}{|H(k_1, N_1, k_2, N_2)|}\sum_{f \ot g \in H(k_1, N_1, k_2, N_2)} D(f \ot g; \phi) = \infint \phi(x) W(Sp(x)) dx.
\end{equation}
If $k_1, k_2 \to \infty$ with $|k_1 - k_2|$ bounded, we can take $\supp \hphi \subset (-1, 1)$.
\end{thm}
In the above theorem, we can take Fourier support $(-1, 1)$ when $|k_1 - k_2|$ is bounded because in this case the analytic conductor for $L(s, f \ot g)$ is small. When $|k_1 - k_2|$ is bounded, the conductor is of size $ k_1^2$. When $|k_1 - k_2|$ is unbounded, the conductor can be as large as $\max(k_1^4, k_2^4)$.

As noted in \cite{ILS}, an application of 1-level density results is to lower bound the proportion of $L$-functions in a given family which do not vanish at the central point. We use the test function
\begin{equation}
    \phi(x) = \left(\frac{\sin(\pi\sigma x)}{\pi\sigma x}\right)^2,
\end{equation}
which has Fourier transform
\begin{equation}
    \hphi(y) = \begin{cases}
        \frac{1}{\sigma} - \frac{|y|}{\sigma^2} & |y| < \sigma \\
        0 & |y| \ge \sigma.
    \end{cases}
\end{equation}
Note that this test function is not optimal for bounding order of vanishing when $\sigma > 1$, so that the constants in \eqref{eq: 0.84} and \eqref{eq: 841} can be slightly improved.
See \cite{boldyriew2023determining, carneiro2022hilbert, dutta2022bounding} for work on optimal test functions. As shown in \cite{ILS}, we have that the proportion of $L$-functions which do not vanish at the central point is lower bounded by
\begin{equation}
    \begin{cases}
        \frac{5}{4} - \frac{1}{2\sigma} & \sigma < 1 \\
        1 - \frac{1}{4\sigma^2} & \sigma \ge 1.
    \end{cases}
\end{equation}
Combining this bound with the earlier theorems gives the following result.
\begin{corr}\label{corr: order}
    Assume GRH. Set
    \begin{equation}
        Z(k_1, N_1, k_2, N_2) = \frac{\left|\{f\ot g \in H(k_1, N_1, k_2, N_2)  \ : \ L(1/2, f\ot g) \ne 0\} \right|}{|H(k_1, N_1, k_2, N_2)|}
    \end{equation}
    to be the proportion of $L$-functions in the family which are nonvanishing at the central point. We have that
    \begin{align}
        \liminf_{\substack{N\to \infty \\ k_1 \ne k_2}} Z(k_1, N, k_2, N) &\ge \frac{21}{25} = 0.84 \label{eq: 0.84} \\
        \liminf_{\substack{N \to \infty}} Z(k, N, k, N) &\ge \frac{645}{841} = 0.7669\ldots \label{eq: 841} \\
        \liminf_{\substack{k_1, k_2\to \infty \\ |k_1 - k_2| \mathrm{\,  bounded}}} Z(k_1, N_1, k_2, N_2) &\ge \frac{3}{4} = 0.75 \\
        \liminf_{\substack{k_1N_1k_2N_2 \to \infty }} Z(k_1, N_1, k_2, N_2) &\ge \frac{1}{4} = 0.25. \label{eq: Z R limit}
    \end{align}
\end{corr}
In the last limit we may take any of $k_1, N_2, k_2, N_2$ to infinity. Because our $L$-functions have even functional equation, it is conjectured $Z(k_1, N_1, k_2, N_2) = 1$ in the limit. In fact, this would follow from the conjecture of Keating and Snaith described below. \cite{kowalski2002rankin} show that $Z(k_1, N_1, k_2, N_2)$ is positive in the limit but do not obtain an explicit constant. We are able to give an explicit (but conditional) lower bound.

As was recently demonstrated by Radziwi{\l\l} and Soundararajan \cite{rad2023conditional}, we can apply Corollary \ref{corr: order} to obtain conditional lower bounds for a conjecture of Keating and Snaith \cite{keating2000random} on the distribution of $L(1/2, f\ot g)$. First, set
\begin{equation}
    \mathcal{N} (k_1, N_1, k_2, N_2, \alpha, \beta) = \frac{\left|\left\{ f \ot g \in H(k_1, N_1, k_2, N_2) \ : \ \frac{\log L(1/2, f \ot g) - \frac{1}{2}\log \log R}{\sqrt{\log \log R}} \in (\alpha, \beta) \right\} \right| }{|H(k_1, N_1, k_2, N_2)|}
\end{equation}
where we say that $\log 0 = -\infty$ by convention. The Keating-Snaith conjecture predicts that as $R \to \infty$
\begin{equation}
    \mathcal{N} (k_1, N_1, k_2, N_2, \alpha, \beta) = \frac{1}{\sqrt{2\pi}} \int_\alpha^\beta e^{-x^2/2} dx + o(1).
\end{equation}
In other words, $\log L(1/2, f \ot g)$ should be distributed approximately normally with mean $\frac{1}{2} \log \log R$ and variance $\log \log R$. The mean of the distribution is expected to depend on the symmetry type of the family. For orthogonal families, the mean is predicted to be $-\frac{1}{2} \log \log R$, as opposed to $+ \frac{1}{2} \log \log R$ for our symplectic family. Radziwi{\l\l} and Soundararajan show that assuming GRH, we can obtain lower bounds for the Keating-Snaith conjecture using lower bounds for the proportion of $L$-functions in a family which are non-vanishing at the central point. 
\begin{corr}\label{cor: keating}
    Assume GRH and fix an interval $(\alpha, \beta)$. Then as $R \to \infty$, we have that
    \begin{equation}
        \mathcal{N} (k_1, N_1, k_2, N_2, \alpha, \beta) \ge c_0 \frac{1}{\sqrt{2\pi}} \int_\alpha^\beta e^{-x^2/2} dx + o(1)
    \end{equation}
    where
    \begin{equation}
        c_0 = \begin{cases}
            0.84 & N_1 = N_2 \to \infty\ \mathrm{with}\ k_1 \ne k_2\ \mathrm{fixed} \\
            0.7669 & N_1 = N_2 \to \infty\ \mathrm{with}\ k_1 = k_2\ \mathrm{fixed} \\
            0.75 & k_1, k_2 \to \infty\ \mathrm{with} \ |k_1 - k_2| \ \mathrm{bounded\ and}\ N_1, N_2 \ \mathrm{fixed} \\
            0.25 & \mathrm{in\ general.}
        \end{cases}
    \end{equation}
\end{corr}
Corollary \ref{cor: keating} follows from modifying the methods of \cite{rad2023conditional} to the family $H(k_1, N_1, k_2, N_2)$ and using the constants in Corollary \ref{corr: order}. These modifications are fairly straightforward, so we omit details to avoid replicating their arguments.

The structure of this paper is as follows. In Section \ref{sec:prelim}, we state some important definitions and review several facts about cusp forms from \cite{ILS}. In Section \ref{sec:explicit}, we go over Rankin-Selberg $L$-functions and develop the explicit formula relating the 1-level density to prime sums. In Section \ref{sec:applying}, we apply the Petersson trace formula to average the explicit formula over our family. Then, in Section \ref{sec: limited} we prove Theorems \ref{thm: general result} and \ref{thm: over weight}. The remainder of the paper is devoted to proving Theorem \ref{thm: extendo}. In Section \ref{sec:kloosterman}, we use GRH for Dirichlet $L$-functions to eliminate the Kloosterman sums which arise from the Petersson formula, and carefully analyze the remaining character sums. In the process, we develop new identities related to sums over products of Gauss, Kloosterman, and Ramanujan sums. Then in Section \ref{sec:extendo}, we prove Theorem \ref{thm: extendo} in the case where $k_1 \ne k_2$ by evaluating the Bessel integral in order to obtain a closed form for an off-diagonal term which only contributes for support outside $(-1, 1)$. Lastly, in Section \ref{sec:pole}, we complete the proof of Theorem \ref{thm: extendo} in the case where $k_1 = k_2$ by handling the contribution from the poles.

\subsection*{Acknowledgments}
The author would like to thank Steven J. Miller for supervising this project and
Leo Goldmakher and Simran Khunger for helpful comments.

\section{Preliminaries}\label{sec:prelim}
In this section we go over some basic facts which will be useful later in the paper.
\subsection{Notation}
Throughout this paper, we use the following notation for sums over residue classes:
\begin{align}
    \sum_{a(q)} f(a) &= \sum_{a = 1}^q f(a) \\
    {\sum_{a(q)}}^* f(a) &= \sum_{\substack{a = 1\\ (a, q) = 1}}^q f(a).
\end{align}
Often the restriction to $(a, q) = 1$ is implicit when the summand involves a Dirichlet character modulo $q$, as in this case $f(a) = 0$ if $(a, q) > 1$.
\begin{defn}[Gauss Sums]
    For $\chi$ a character modulo $q$ and $e(x)  =  e^{2\pi i x}$,
    \begin{equation}\label{eq:gausssum}
        G_\chi(n)\ \coloneqq \ \sum_{a (q)} \chi(a) e(an/q).
    \end{equation} 
    By Theorem 9.12 of \cite{MV}, we have that 
    \begin{equation} \label{eq: Gauss sum bound MV}
    |G_\chi(n)| \le (n, q) \sqrt{q}.
    \end{equation}
\end{defn}

\begin{defn}[Ramanujan Sums]\label{def:RamSums}
    If $\chi = \chi_0$ (the principal character modulo $q$) in \eqref{eq:gausssum}, then $G_{\chi_0}(n)$ becomes the Ramanujan sum 
    \begin{equation}\label{eq:defn Rnew} 
        R(n,q) \ \coloneqq \ \sideset{}{^\ast}\sum_{a ( q)} e(an/q) \ = \ \sum_{d|(n,q)} \mu(q/d) d.
    \end{equation} 
\end{defn}

\noindent The Ramanujan sum satisfies the following identity:

\begin{equation}\label{eq:vonsterneck}
    R(n, q)\ = \ \mu \left(\frac{q}{(q, n)} \right) \frac{\varphi(q)}{\varphi\left(\frac{q}{(q, n)} \right)}.
\end{equation}

\begin{defn}[Kloosterman Sums] 
    For integers $m$ and $n$,
    \begin{equation}
        S(m,n;q)\ \coloneqq \ \sideset{}{^\ast}\sum_{d (q)} e\left(\frac{md}{q} + \frac{n\overline{d}}{q}\right),
    \end{equation}
    where $d \overline{d} \equiv 1 \bmod q$. 
\end{defn}    
\noindent The Kloosterman sum satisfies the Weil bound
    \begin{equation}\label{eq:estimate Kloosterman}
        |S(m,n;q)| \ \leq \ (m,n,q)\ \sqrt{\min\left\{ \frac{q}{(m,q)} , \frac{q}{(n,q)} \right\}}\ \ \tau(q),
    \end{equation}
    where $\tau(q)$ is the number of positive divisors of $q$; see Equation 2.13 of \cite{ILS}.

\begin{defn}[Fourier Transform] 
    We use the following normalization:
    \begin{equation}
        \hphi(y)\ \coloneqq \ \infint \phi(x) e^{-2\pi ixy} \;d x, \ \ \ \ \ \phi(x)\ \coloneqq \ \infint \hphi(y) e^{2\pi i xy} \;d y.
    \end{equation}
\end{defn}

\begin{defn}[(Infinite) GCD]
    For $x,y \in \Z$, let $(x, y)$ denote the greatest common divisor of $x$ and $y$. Set $(x, y^{\infty}) = \sup_{n \in \N} (x, y^{n})$ and $(x^{\infty}, y) = \sup_{n \in \N}(x^{n}, y)$.
\end{defn}

The Bessel function of the first kind occurs frequently in this paper, and so we collect here some standard bounds for it from (2.11) of \cite{ILS} and Lemma 2.6 of \cite{HM}.

\begin{lemma}\label{lem:Bessel} 
    Let $k\geq 2$ be an integer. The Bessel function satisfies
    \begin{enumerate}
        \item $J_{k -1}(x) \ll 1$,
        \item $J_{k-1}(x) \ll x^{-1/2}$,
        \item\label{lb:1} $J_{k-1}(x) \ll \min\left(1, \frac{x}{k}\right)k^{-1/3}$,
        \item\label{lb:2} $J_{k-1}(x) \ll x 2^{-k},\quad 0 < x \le \frac{k}{3}$.
    \end{enumerate}
\end{lemma}
Throughout the paper, there will be a special case when the weights and levels of the two families we convolve are the same, as in this case some of the $L$-functions in the family have poles. We use the following indicator function to indicate this.
\begin{defn}
    We have
    \begin{equation}
        \dpole \clq \dpole(k_1, N_1, k_2, N_2) \clq \begin{cases}
            1 & (k_1, N_1) = (k_2, N_2) \\
            0 & \text{otherwise.}
        \end{cases}
    \end{equation}
\end{defn}
\subsection{Cusp forms}
We recall some important facts about cusp forms from \cite{ILS}. For more information on cusp forms, see \cite{iwaniec1997topics, ono2004web,diamond2005modular}. Let $S_k(N)$ denote the set of cusp forms of even weight $k$ and prime level $N$ for the Hecke congruence subgroup $\Gamma_0(N)$. Note that we assume the level $N$ is prime, but most of the arguments should hold for squarefree level $N$ as in \cite{ILS}. These are cusp forms for the congruence subgroup $\Gamma_1(N)$ with trivial nebentypus (central character). Each $f \in S_k(N)$ has Fourier expansion
\begin{equation}
    f(z) = \sum_{n=1}^\infty a_f(n) e(nz).
\end{equation}
If $f$ is a newform, then it is a fact that $a_f(1) \ne 0$, so we can normalize $f$ so that $a_f(1) = 1$. Let $H_k^*(N)$ denote the set of $f \in S_k(N)$ which are newforms of level $N$ normalized so that $a_f(1) = 1$. 
\begin{lemma}[\cite{ILS}, Corollary 2.14]
Let $H_k^*(N)$ denote the set of normalized cusp forms which are newforms of level $N$. Then
\begin{equation}\label{eq: H size}
    |H_k^*(N)| = \frac{k-1}{12}\varphi(N) + O\left((kN)^{2/3}\right).
\end{equation}
\end{lemma}
\noindent
Now, set
\begin{equation}
    \lambda_f(n) \coloneqq a_f(n)n^{-(k-1)/2}.
\end{equation}
Each $f \in H_k^*(N)$ is an eigenfunction of all the Hecke operators $T_n$ with eigenvalue $\lambda_f(n)$. The Hecke eigenvalues are multiplicative, and in particular we have
\begin{equation}
    \lambda_f(m)\lambda_f(n) = \sum_{\substack{d|(m,n) \\ (d, N) = 1}} \lambda_f\left(\frac{mn}{d^2}\right)
\end{equation}
so that $\lambda_f(m)\lambda_f(n) = \lambda_f(mn)$ if $(m,n) = 1$.

\subsection{Petersson trace formula}

Essential to our results will be the Petersson trace formula \cite{petersson1932}, which allows us to calculate averages over Fourier coefficients. Set
\begin{equation}\label{eq: psi coeffs}
    \psi_f(n) \coloneqq \left(\frac{\Gamma(k-1)}{(4\pi n)^{k-1}} \right)^{1/2} ||f||\inv a_f(n)
\end{equation}
where $||f||^2 = \langle f, f \rangle $ is the Petersson inner product on $S_k(N)$, defined as
\begin{equation}
    \langle f, g \rangle = \int_{\Gamma_0(N) \backslash H} f(z) \overline{g}(z) y^{k-2} dx dy, \quad z = x + iy.
\end{equation}
Next, put
\begin{equation}
    \Delta_{k, N}(m,n) \coloneqq \sum_{f \in \mathcal{B}_k(N)} \overline{\psi_f(m)} \psi_f(n)
\end{equation}
where the sum is over an orthogonal basis $\mathcal{B}_k(N)$ for $S_k(N)$. The classical Petersson formula gives

\begin{equation}\label{eq: Petersson classic}
    \Delta_{k, N}(m,n) = \delta(m, n) + 2\pi i^k \sum_{b=1}^\infty \frac{S(m, n; bN)}{bN} J_{k-1}\left( \frac{4\pi \sqrt{mn}}{bN} \right).
\end{equation}
\cite{ILS} gives another version of the Petersson formula which will be useful later.

\begin{lemma}[\cite{ILS}, Lemma 2.7]\label{lem: ILS 2.7}
Set
\begin{equation}\label{eq:nu def}
    \nu(N) \coloneqq \left[\Gamma_0(1) : \Gamma_0(N)\right] = N \prod_{p|N}\left(1 + p\inv\right)
\end{equation}
and define the following zeta functions
\begin{equation}\label{eq: Z def}
     Z(s,f) \clq \sum_{n=1}^\infty \lambda_f(n^2) n^{-s}, \qquad Z_N(s,f) \clq \sum_{n|N^\infty} \lambda_f(n^2)n^{-s}.
\end{equation}
Let $(m,n, N) = 1$ and $(mn, N^2) | N$. Then
\begin{equation}\label{eq: ILS 2.7}
    \Delta_{k, N}(m,n) = \frac{12}{(k-1)N} \sum_{LM = N} \sum_{f \in H_k^*(M)} \frac{\lambda_f(m)\lambda_f(n)}{\nu((mn, L))} \frac{Z_N(1,f)}{Z(1,f)}.
\end{equation}
\end{lemma}
Of particular interest to us are the pure sums
\begin{equation}
    \Delta_{k, N}^*(n) \clq \sum_{f\in H^*_k(N)} \lambda_f(n).
\end{equation}
We use the following result from \cite{ILS}.
\begin{lemma}[\cite{ILS}, Proposition 2.11]\label{lem: ILS 2.11}
If $(n, N^2)|N$, then
    \begin{equation}\label{eq: Petersson unweighted}
    \Delta^*_{k, N}(n) = \frac{k-1}{12} \sum_{LM = N} \frac{\mu(L) M}{\nu((n, L))} \sum_{(m, M) = 1} m\inv \Delta_{k, M}(m^2, n).
\end{equation}
\end{lemma}
\begin{rem}\label{rem: MN}
    In our case when $N$ is prime, the main contribution to \eqref{eq: ILS 2.7} and \eqref{eq: Petersson unweighted} comes from when $M = N$. If we fix $k$ and take $N \to \infty$, then the $M=1$ term is $O(1)$, and it is clear from our application of the Petersson formula in Section \ref{sec:applying} (see \eqref{eq: sum over f} and \eqref{eq: sum over g} in particular) that the $M=1$ term does not contribute in the limit. If we take $k \to \infty$ as in Theorem \ref{thm: over weight}, we show in Section \ref{sec: limited} that the $M= N$ term vanishes, and showing that the $ M = 1$ term vanishes is nearly identical. Thus for the remainder of the paper we will ignore the $M = 1$ term for simplicity.
\end{rem}

We split the sum into two pieces as
\begin{equation}\label{eq: delta split}
    \Delta_{k, N}^*(n) = \Delta_{k, N}'(n) + \Delta_{k, N}^\infty(n)
\end{equation}
where
\begin{equation}\label{eq: delta star}
    \Delta'_{k, N}(n) = \frac{(k-1)N}{12} \sum_{\substack{(m, N) = 1 \\ m \le Y}} m\inv \Delta_{k, N}(m^2, n)
\end{equation}
and $\Delta^\infty_{k, N} (n)$ is the complementary sum (the terms with $m > Y$). Here, $Y$ is a parameter which we set to $(k_1k_2N_1N_2)^{4\epsilon}$ in Section \ref{sec: limited}.
\begin{rem}
    \cite{ILS} introduces an additional parameter $X$ into the sums $\Delta'$ and $\Delta^\infty$, which we avoid by eliminating the terms with $M \ne N$ (see Remark \ref{rem: MN}). This means we may take any $X \ge 1$ when using bounds from \cite{ILS}. Using $X = (k_1k_2N_1N_2)^\epsilon$ is sufficient.
\end{rem}

\section{The explicit formula}\label{sec:explicit}
In this section we develop the explicit formula for Rankin-Selberg $L$-functions to relate the 1-level densities to sums over Fourier coefficients. Many of our results about Rankin-Selberg $L$-functions come from \cite{li1979series} and Section 4 of \cite{kowalski2002rankin}. 

\subsection{Convolution $L$-functions}\label{sec: RSL}
We consider two families of cusp forms $\Hone$ and $\Htwo$, both with even weights $k_1$ and $k_2$ and prime levels $N_1$ and $N_2$.

Let $f \in \Hone$ and $g \in \Htwo$. We are interested in studying the convolution $f \ot g$. As the forms in our original family are self--dual, the convolution $f \ot g$ is as well.
The Rankin-Selberg convolution $L$-function is 
\begin{align}
    L(s, f \ot g) :&= L(2s, \chi_{0}^{N_1N_2})\sum_{n \ge 1} \frac{\lambda_f(n) \lambda_g(n)}{n^s}\\
    &=  \prod_p \prod_{i=1}^2 \prod_{j=1}^2 \left( 1 - \alpha_{f, i}(p) \alpha_{g, j}(p) p^{-s} \right)\inv,\nn
\end{align}
where $\chi_{0}^N$ denotes the principal character modulo $N$ and $\alpha_{f, i}$ are the roots of the equation
\begin{equation}
    x^2 - \lambda_f(p) x + \chi_{0}^{N_1}(p) = 0.
\end{equation}
The analogous definition holds for $\alpha_{g, j}(p)$. We set 
\begin{equation}
    L_\infty(s, f\ot g) := \left( \frac{[N_1, N_2]}{4\pi^2} \right)^s \Gamma\left(s +  \frac{|k_1 - k_2|}{2} \right) \Gamma\left(s + \frac{k_1 + k_2}{2} - 1\right).
\end{equation}
By the duplication formula for the gamma function, we can write
\begin{align}\label{eq: L infty}
    L_\infty(s, f\ot g) &= \left(\frac{[N_1, N_2]}{\pi^2} \right)^s \frac{2^{\max(k_1, k_2)}}{8\pi} \Gamma\left( \frac{s}{2} + \frac{|k_1 - k_2|}{4} \right) \Gamma\left( \frac{s}{2} + \frac{|k_1 - k_2| + 2}{4} \right)\nn \\
    &\hspace{1cm}\times \Gamma\left( \frac{s}{2} + \frac{k_1 + k_2 - 2}{4} \right) \Gamma\left( \frac{s}{2} + \frac{k_1 + k_2}{4} \right). 
\end{align}
The completed $L$-function is
\begin{equation}
    \Lambda(s, f \ot g) :=  L_\infty(s, f \ot g)L(s, f \ot g).
\end{equation}
It satisfies the functional equation
\begin{equation}
    \Lambda(s, f \ot g) = \Lambda(1-s, f \ot g).
\end{equation}

\subsection{Explicit formula}
Let $\phi$ be an even test function whose Fourier transform is compactly supported in some fixed interval $(-\sigma, \sigma)$. Set
\begin{equation}
    D(f \ot g; \phi) := \sum_{\rho_{f \ot g}} \phi\left( \frac{\gamma_{f \ot g}}{2\pi} \log R \right)
\end{equation}
as in \eqref{eq: density def}, where the sum is over the nontrivial zeros $\rho_{f \ot g} = \frac{1}{2} + i\gamma_{f \ot g}$ of $L(s, f \ot g)$. $R$ is a normalization factor which we set to the analytic conductor of our $L$-functions. The conductor is 
\begin{equation}\label{eq: Q def}
    R = \begin{cases} [N_1, N_2]^2 (k_1 - k_2)^2 (k_1 +k_2)^2 & k_1 \ne k_2 \\
    [N_1, N_2]^2 k_1^2 & k_1 = k_2
    \end{cases}
\end{equation}
which comes from the gamma factors in \eqref{eq: L infty}. The conductor naturally appears in the explicit formula \eqref{eq: explicit formula alphas} and \eqref{eq: A def}. 

We derive the explicit formula as in Section 4 of \cite{ILS}. We apply the argument principle to $\Lambda(s, f\ot g)$ multiplied by the normalized test function
\begin{equation}
    \phi\left(\left(s - \frac{1}{2} \right) \frac{\log R}{2\pi i} \right).
\end{equation}
If $f \ne \overline{g}$ then by (4.11) of \cite{ILS} we have
\begin{equation}\label{eq: explicit formula alphas}
    D(f \ot g; \phi) = \frac{A}{\log R} - 2 \sum_p \sum_{\nu = 1}^\infty  \left(\sum_{i,j} \alpha_{f, i}^\nu(p) \alpha_{g, j}^\nu(p) \right) \hphi\left(\frac{\nu \log p }{\log R} \right) p^{-\nu/2} \frac{\log p}{\log R}
\end{equation}
where
\begin{equation}\label{eq: A def}
    A = \hphi(0)\log R +  O(1).
\end{equation}
If $ f = \overline{g}$, there is an additional term from a pole at $ s= 1$. The contribution of the pole is $2\phi\left( \frac{\log R}{4\pi i} \right)$, where we extend the definition of $\phi$ to $\C$ using the inverse Fourier transform:
\begin{equation}
    \phi(z) = \infint \hphi(y) e^{2\pi i zy} dy, \quad z \in \C.
\end{equation}

By the Ramanujan conjectures for $f$ and $g$, we have that $|\alpha_{f,i}(p)|, |\alpha_{g, j}(p)| \le 1$, so the terms with $\nu \ge 3$ in \eqref{eq: explicit formula alphas} are $\OlogR$. For the terms with $\nu = 1$, we have that
\begin{align}
    \sum_{i,j} \alpha_{f, i}(p) \alpha_{g, j}(p) &= \left(\sum_{i} \alpha_{f, i}(p) \right) \left( \sum_j \alpha_{g,j} (p) \right) \\
    &= \lambda_f(p) \lambda_g(p). \nn
\end{align}
For the $\nu = 2$ terms, we have that
\begin{align}
    \sum_{i,j} \alpha_{f, i}^2(p) \alpha_{g, j}^2(p) &= \left(\sum_{i} \alpha_{f, i}^2(p) \right) \left( \sum_j \alpha_{g,j}^2 (p) \right) \\
    &= \left(\lambda_f(p^2) - \chi_{0}^{N_1}(p)\right) \left(\lambda_g(p^2) - \chi_{0}^{N_2}(p)\right). \nn
\end{align}
Putting this all together we have that
\begin{align}\label{eq: D explicit formula complete}
    D(f \ot g; \phi) &= \hphi(0) - \sum_p \lambda_f(p)\lambda_g(p) \hphi\left(\frac{\log p}{\log R}\right) \frac{2\log p}{\sqrt{p} \log R} \\
    &\hspace{1cm} - \sum_p \left(\lambda_f(p^2)\lambda_g(p^2) - \lambda_f(p^2) - \lambda_g(p^2)\right) \hphi\left(\frac{2\log p}{\log R}\right) \frac{2\log p}{p \log R}\nn \\
    &\hspace{1cm} - \sum_p \hphi\left(\frac{2\log p}{\log R}\right) \frac{2\log p}{p \log R} + 2\delta(f, \overline{g}) \phi\left( \frac{\log R}{4\pi i} \right) + \OlogR, \nn
\end{align}
where $\delta(f, \overline g) = 1$ if $f = \overline g$ and 0 otherwise.
We use that $\chi_{0}^{N_1}(p) = 1$ if $p \ne N_1$ and 0 otherwise, and the terms with $p = N_1$ or $p = N_2$ are trivially absorbed into the error term. 

We have that
\begin{align}\label{eq: sym square defs}
    L(s, \sym^2 f) &= L(2s, \chi_0^{N_1}) \sum_{n \ge 1} \frac{\lambda_f(n^2)}{n^s} \\
    L(s, \sym^2 g) &= L(2s, \chi_0^{N_2}) \sum_{n \ge 1} \frac{\lambda_g(n^2)}{n^s} \\
    L(s, \sym^2 f \ot \sym^2 g) &= V(s, f, g) \sum_{n \ge 1} \frac{\lambda_f(n^2)\lambda_g(n^2)}{n^s} 
\end{align}
where $V(s,f,g)$ is an Euler product converging absolutely for $\Re(s) > 1/2$ (\cite{ILS}, (3.20)). Thus the explicit formula and GRH for $L(s, \sym^2(f))$, $L(s, \sym^2(g))$, and $L(s, \sym^2(f)\ot \sym^2(g))$ gives bounds for prime sums over $\lambda_f(p^2)$, $\lambda_g(p^2)$, and $\lambda_f(p^2)\lambda_g(p^2)$, respectively (\cite{ILS}, (4.23) and (4.24)). In particular, assuming GRH, we have that the second sum in \eqref{eq: D explicit formula complete} is $O(\log \log R/\log R)$ if $f \ne \overline{g}$. If $f = \overline{g}$, then the second sum is $O(1)$ (but it will still vanish after averaging over the family). Lastly, we have by the prime number theorem and partial summation that
\begin{equation}
    \sum_p \hphi\left(\frac{2\log p}{\log R}\right) \frac{2\log p}{p \log R} = \frac{1}{2}\phi(0) + \OlogR.
\end{equation}
Combining these bounds gives the main result of the section.
\begin{prop}\label{prop: explicit formula}
We have
\begin{equation}\label{eq: explicit formula final}
    D(f \ot g; \phi) = \hphi(0)  - \frac{1}{2 }\phi(0)  - S(f\ot g; \phi) + 2 \delta(f, \overline g) \phi\left( \frac{\log R}{4\pi i} \right) + O\left(\frac{\log \log R}{\log R} + \delta(f, \overline g) \right) 
\end{equation}
where
\begin{equation}\label{eq:S def}
    S(f \ot g; \phi) \clq \sum_p \lambda_f(p)\lambda_g(p) \hphi\left(\frac{\log p}{\log R}\right) \frac{2\log p}{\sqrt{p} \log R}.
\end{equation}
\end{prop}
\begin{rem}\label{rem: DM limit}
    Up to this point our methodology has been the same as in \cite{DM}. To make the results of their paper explicit, we first sum \eqref{eq:S def} over $f$ and $g$:
    \begin{equation}\label{eq: stupid sum}
        \sum_{f \ot g \in H(k_1, N_1, k_2, N_2)} S(f \ot g; \phi) = \sum_p \left[\sum_{f \in H_{k_1}^*(N_1)} \lambda_f(p)\right] \left[\sum_{g \in H_{k_2}^*(N_2)} \lambda_g(p)\right] \hphi\left(\frac{\log p}{\log R}\right) \frac{2\log p}{\sqrt{p} \log R}.
    \end{equation}
    Proposition 2.13 of \cite{ILS} gives when $p \ne N$ that
    \begin{equation}\label{eq: ILS prop 2.13}
        \sum_{f \in H_{k}^*(N)} \lambda_f(p) \ll p^{1/6}(kN)^{2/3}.
    \end{equation}
    Applying this bound and averaging over the size of the family using \eqref{eq: H size} gives
    \begin{equation}
        \frac{1}{|H(k_1, N_1, k_2, N_2)|} \sum_{f \ot g \in H(k_1, N_1, k_2, N_2)} S(f \ot g; \phi) \ll R^{5\sigma/6}(k_1N_1k_2N_2)^{-1/3}.
    \end{equation}
    We need the above sum to vanish in the limit when $\sigma$ is sufficiently small. Fix $k_1, k_2$ and let $N_1, N_2 \to \infty$. Using the estimate $R \ll (N_1N_2)^2$ gives the density conjecture when $\sigma < 1/5$. In the ``small conductor'' case when $N_1 = N_2$, we have $R \ll N_1N_2$ so we can take $\sigma < 2/5$. 

    In the approach outlined above, the sums over Fourier coefficients in \eqref{eq: stupid sum} are treated separately using the bound \eqref{eq: ILS prop 2.13}. To prove our main theorems, we treat the sums in conjunction in the following sections.
\end{rem}

\section{Applying the Petersson formula}\label{sec:applying}
In this section we use the Petersson formula to average $S(f \ot g; \phi)$ over the forms in the convolved family. Our main result is the following.

\begin{prop}\label{prop: Petersson final}
We have that
\begin{align}\label{eq: Petersson}
    &\sum_{f \ot g \in H(k_1, N_1, k_2, N_2)} S(f \ot g; \phi)\nn \\
    &\hs{1} \clq \frac{(k_1 - 1)}{12} \frac{(k_2 - 1)}{12}4\pi^2i^{k_1 + k_2} \sum_{m_1, m_2 \le Y} \frac{1}{m_1m_2} \sum_{b_1, b_2  \ge 1} \frac{1}{b_1b_2} Q^*(m_1^2,b_1N_1,m_2^2, b_2N_2)\nn\\
    & \hs{2} + Y^{-1/2 + \epsilon} R^\epsilon O\left( k_1N_1k_2N_2 + \dpole  k_1N_1 R^{\sigma/2} \right)
\end{align}
where
\begin{align}
    Q^*(m_1^2,c_1,m_2^2, c_2) =& \sum_p S(m_1^2, p; c_1) S(m_2^2, p; c_2) J_{k_1-1}\left(\frac{4\pi m_1\sqrt p}{c_1 }\right) J_{k_2-1}\left(\frac{4\pi m_2\sqrt p}{c_2 }\right)\nn \\
    &\times \frac{2 \log p}{\sqrt p \log R} \hphi\left(\frac{\log p}{\log R} \right).
\end{align}
\end{prop}
\begin{proof}
We first sum over $f$. To do so, we use \eqref{eq: delta split} to find
\begin{equation}\label{eq: sum over f}
    \sum_{f \in H_{k_1}^*(N_1)} S(f \ot g; \phi) = \sum_p (\Delta'_{k_1, N_1}(p) + \Delta^\infty_{k_1, N_1}(p))\lambda_g(p) \hphi\left(\frac{\log p}{\log R}\right) \frac{2\log p}{\sqrt{p} \log R}. 
\end{equation}
We use a modification of Lemma 2.12 from \cite{ILS} to bound away the complementary sum.
\begin{lemma}[\cite{ILS}, Lemma 2.12]\label{lem: ILS 2.12}
    Assume GRH. Set
    \begin{equation}\label{eq: S1}
    S_1^\infty = \sum_p  \Delta^\infty_{k_1, N_1}(p) \lambda_g(p) \hphi\left(\frac{\log p}{\log R}\right) \frac{2\log p}{\sqrt{p} \log R}. 
\end{equation} 
    Then
    \begin{equation}
        S_1^\infty \ll Y^{-1/2 + \epsilon} R^\epsilon \left[ k_1N_1 +\dpole R^{\sigma/2} \right].
    \end{equation}
\end{lemma}
\begin{proof}
    Expanding $\Delta^\infty_{k_1, N_1}(p)$ and applying Lemma \ref{lem: ILS 2.7} gives
    \begin{align}
        S_1^\infty \ll \sum_{f \in H_{k_1}^*(N_1)} \left[\frac{Z_{N_1}(1,f)}{Z(1,f)} \sum_{m_1 > Y} m_1\inv \lambda_f(m_1^2)  \right] \left[ \sum_p \frac{\lambda_f(p)\lambda_g(p) \log p}{\sqrt p \log R} \hphi\left(\frac{\log p}{\log R} \right) \right].
    \end{align}
    By GRH for $L(s, \sym^2 f)$, the first term in brackets is $\ll Y^{-1/2} (k_1 N_1 Y)^\epsilon$ \cite[Proof of Lemma 2.12]{ILS}. In particular, this follows from the Lindel\"of hypothesis for $L(s, \sym^2 f)$ and \eqref{eq: sym square defs}; see \cite[(5.61)]{IK}. If $f \ne \overline{g}$, then GRH for $L(s, f\ot g)$ gives that the sum over $p$ is $\ll R^\epsilon $. If $f = \overline{g}$, then the sum over $p$ is of size $R^{\sigma/2}$, which only occurs if $(k_1, N_1) = (k_2, N_2)$. Combining these bounds gives the lemma.
\end{proof}

Applying this lemma to \eqref{eq: sum over f} gives
\begin{align}
    \sum_{f \in H_{k_1}^*(N_1)} S(f \ot g; \phi) &= \sum_p \Delta'_{k_1, N_1}(p)\lambda_g(p) \hphi\left(\frac{\log p}{\log R}\right) \frac{2\log p}{\sqrt{p} \log R} \nn\\
    &\hs{1} + Y^{-1/2 + \epsilon} R^\epsilon O\left(k_1N_1 + \dpole R^{\sigma/2} \right).
\end{align}
Next we want to sum over $g$. Doing so gives
\begin{align}\label{eq: sum over g}
    \sum_{f \ot g \in H(k_1, N_1, k_2, N_2)} S(f \ot g; \phi) &= \sum_p \Delta'_{k_1, N_1}(p)\Delta'_{k_2, N_2}(p)  \hphi\left(\frac{\log p}{\log R}\right) \frac{2\log p}{\sqrt{p} \log R}\\
    &\hspace{1cm} + \sum_p \Delta'_{k_1, N_1}(p) \Delta^\infty_{k_2, N_2}(p)  \hphi\left(\frac{\log p}{\log R}\right) \frac{2\log p}{\sqrt{p} \log R}\nn\\
    &\hspace{1cm} + Y^{-1/2 + \epsilon} R^\epsilon O\left(k_1N_1k_2N_2+ \dpole k_1N_1 R^{\sigma/2} \right). \nn
\end{align}
The first sum is the main term, and we use a method similar to Lemma \ref{lem: ILS 2.12} to show that the second sum vanishes in the limit. 
\begin{lemma}\label{lem: double complementary sum}
Assume GRH. Set 
\begin{equation}
    S^\infty_2 \coloneqq \sum_p \Delta'_{k_1, N_1}(p) \Delta^\infty_{k_2, N_2}(p)  \hphi\left(\frac{\log p}{\log R}\right) \frac{2\log p}{\sqrt{p} \log R}.
\end{equation}
Then
\begin{equation}
    S^\infty_2 \ll Y^{-1/2 + \epsilon}R^ {\epsilon} \left[ k_1 N_1 k_2 N_2 + \dpole k_1N_1 R^{\sigma/2} \right].
\end{equation}
\end{lemma}
\begin{proof}
Expanding $\Delta'_{k_1, N_1}(p)$ and $\Delta^\infty_{k_2, N_2}(p)$ using \eqref{eq: delta split} and \eqref{eq: delta star} gives
\begin{align}
    S^\infty_2 &\ll \sum_p N_1k_1\sum_{\substack{m_1 \le Y \\ (m_1, N_1) = 1}} m_1\inv \Delta_{k_1, N_1}(m^2_1, p) N_2k_2 \sum_{\substack{m_2 > Y \\ (m_2, N_2) = 1}} m_2\inv \Delta_{k_2, N_2}(m_2^2, p) \\
    &\hs{1} \times\hphi\left(\frac{\log p}{\log R}\right) \frac{\log p}{\sqrt p \log R}.\nn
\end{align}
Next we apply Lemma \ref{lem: ILS 2.7} and rearrange to give
\begin{align}\label{eq: infinity sum rearrange}
    S^\infty_2 &\ll \sum_{f \ot g \in H(k_1, N_1, k_2, N_2)} \left[ \frac{Z_{N_1}(1,f)}{Z(1, f)} \sum_{\substack{m_1 \le Y \\ (m_1, N_1) = 1}} m_1\inv \lambda_f(m_1^2) \right] \left[ \frac{Z_{N_2}(1,g)}{Z(1, g)} \sum_{\substack{m_2 > Y \\ (m_2, N_2) = 1}} m_2\inv \lambda_g(m_2^2) \right]\nn \\
    &\hs{1} \times \sum_p \frac{\lambda_f(p)\lambda_g(p) \log p}{\sqrt p \log R} \hphi\left(\frac{\log p}{\log R} \right).
 \end{align}

Arguing as in the proof of Lemma \ref{lem: ILS 2.12}, by GRH for $L(s, \sym^2(f))$, the first sum in brackets is $\ll Y^{-1/2} (k_1N_1Y)^\epsilon$. Likewise, by GRH for $L(s, \sym^2(g))$, the second sum in brackets is $\ll Y^{-1/2}(k_2N_2Y)^\epsilon$. Lastly, if $f \ne \overline{g}$, by GRH for $L(s, f\ot g)$ the sum over $p$ is $\ll (k_1k_2N_1N_2)^\epsilon $. If $f = \overline{g}$, then the sum over $p$ is of size $R^{\sigma/2}$. If $(k_1, N_1) = (k_2, N_2)$, there will be $|H_{k_1}(N_1)| \asymp k_1N_1$ terms in the sum for which $f = \overline{g}$, so their contribution is $k_1N_1 R^{\sigma/2}$. Combining these bounds gives the lemma.
\end{proof}

\begin{rem}\label{rem: pole fail}
    When $f = \overline{g}$, the lack of square root cancellation in the prime sum in \eqref{eq: infinity sum rearrange} means that $S^\infty_2$ contributes to the main term when $\hphi$ is supported outside $(-1, 1)$ and $Y = (k_1k_2N_1N_2)^\epsilon$. In Section \ref{sec:pole}, we account for this by taking $Y = N^{\alpha}$ with $\alpha = 1/14$. However, in this case $m_1$ and $m_2$ have non-negligible size, which requires more careful bounding of the main term sums over $\Delta'$.
\end{rem}

Lemma \ref{lem: double complementary sum} shows that $S_2^\infty$ is absorbed by the error term in \eqref{eq: sum over g}. Now, we want to expand the $\Delta'$s using the Petersson formula. By \eqref{eq: Petersson classic} and \eqref{eq: delta star} we have
\begin{equation}
    \Delta'_{k, N}(p) = \frac{N(k-1)}{12} 2\pi i^k \sum_{m \le Y} \frac{1}{m} \sum_{b= 1}^\infty \frac{S(m^2, p; bN)}{bN} J_{k-1}\left(\frac{4\pi m\sqrt p}{bN }\right).
\end{equation}
Applying this to \eqref{eq: sum over g} completes the proof of Proposition \ref{prop: Petersson final}. 
\end{proof}

\section{Proofs of Theorems 1.1 and 1.4}\label{sec: limited}
In this section, we use Proposition \ref{prop: Petersson final} to complete the proofs of Theorems \ref{thm: general result} and \ref{thm: over weight}. First we need to account for the contribution from any potential poles. We can bound the contribution from the poles to \eqref{eq: explicit formula final} by
\begin{equation}\label{eq: pole bound}
    |H_{k_1}^*(N_1)| \phi\left( \frac{\log R}{4\pi i} \right) =|H_{k_1}^*(N_1)| \infint \hphi(y) R^{y/2} dy \ll k_1N_1 R^{\sigma/2}.
\end{equation}
A pole occurs only if $(k_1, N_1) = (k_2, N_2)$, in which case $R = k_1^2 N_1^2$. After dividing by the size of the family using \eqref{eq: H size}, we find that the contribution from the poles is $O\left(k_1^{\sigma - 1} N_1^{\sigma - 1} \right)$,
which vanishes in the limit if $\sigma < 1$. Note that if $\sigma \ge 1$, we cannot bound away the contribution from the pole. In Section \ref{sec:pole}, we show that a new term emerges from the average over $S(f \ot g; \phi)$ which cancels the contribution from the pole when $\sigma \ge 1$.

Because Theorems \ref{thm: general result} and \ref{thm: over weight} require that $\hphi$ be supported in $(-1, 1)$, we need to show that
\begin{equation}
    \frac{1}{|H(k_1, N_1, k_2, N_2)|} \sum_{f \ot g \in H(k_1, N_1, k_2, N_2)} S(f \ot g; \phi) = o(1)
\end{equation}
and then the theorems will follow from Proposition \ref{prop: explicit formula}, \eqref{eq: H size}, and comparing with \eqref{eq: W hat}.

We now treat Theorems \ref{thm: general result} and \ref{thm: over weight} separately.

\begin{proof}[Proof of Theorem \ref{thm: general result}]
    Fix $k_1, k_2$. We bound $Q^*$ with \eqref{eq:estimate Kloosterman} and $J_{k-1}(x) \ll x$, giving
\begin{equation}\label{eq: easy Q bound}
    Q^*(m_1^2,c_1,m_2^2, c_2) \ll (m_1m_2c_1c_2)^\epsilon m_1m_2 (c_1c_2)^{-1/2} R^{3\sigma/2}.
\end{equation}
    Applying \eqref{eq: easy Q bound} to \eqref{eq: Petersson} and setting $Y = (N_1N_2)^{5\epsilon}$ (so that the $Y^{-1/2+\epsilon}$ term appearing in \eqref{eq: Petersson} decays sufficiently quickly) gives
\begin{align}\label{eq: 1.1 bound}
    \sum_{f \ot g \in H(k_1, N_1, k_2, N_2)} S(f \ot g; \phi) &\ll  (N_1N_2)^{-1/2} R^{3\sigma/2} (N_1N_2)^\epsilon    + R^{-\epsilon}\left( N_1N_2+ \dpole  N_1R^{\sigma/2}\right).
\end{align}
Recall that if $k_1, k_2$ are fixed, then $R \asymp N_1^2N_2^2$ if $N_1 \ne N_2$ and $R \asymp N_1^2$ if $N_1 = N_2$.
If $\sigma <1/2$, we have that \eqref{eq: 1.1 bound} is $\ll (N_1N_2)^{1-\epsilon}$ (note that $\dpole = 0$ unless $N_1 = N_2$). By \eqref{eq: H size}, we have that our family is of size $\asymp N_1N_2$, which completes the proof of Theorem \ref{thm: general result}.
\end{proof}

\begin{rem}
    If $N_1 = N_2$, we can take $\sigma < 1$ in the above proof. However, this result is superseded by Theorem \ref{thm: extendo}.
\end{rem}

\begin{proof}[Proof of Theorem \ref{thm: over weight}]
Fix $N_1, N_2$. We argue as in the proof of Theorem \ref{thm: general result} above, but instead use the stronger bound $J_{k-1}(x) \ll x2^{-k}$ from Lemma \ref{lem:Bessel}, which holds when $x < k/3$. To use this bound, we need one of the following two inequalities to hold (we can just use $J_{k-1}(x) \ll x$ for the other one):
\begin{align}
    R^{\sigma/2} &\ll k_1N_1Y^{-1} \label{eq: strong bessel cond}\\
    R^{\sigma/2} &\ll k_2N_2Y^{-1} \label{eq: strong bessel cond2}.
\end{align}
If $|k_1 - k_2|$ is bounded by an absolute constant, then $R \ll k_1^2$ and $R \ll k_2^2$, so \eqref{eq: strong bessel cond} and \eqref{eq: strong bessel cond2} hold when $\sigma < 1$. If $|k_1 - k_2|$ is unbounded, then $R \ll k_1^4$ or $R \ll k_2^4$, so that one of \eqref{eq: strong bessel cond} or \eqref{eq: strong bessel cond2} holds when $\sigma < 1/2$. Thus assuming the hypothesis of Theorem \ref{thm: over weight}, without loss of generality we have that \eqref{eq: strong bessel cond} holds.
Bounding $Q^*$ with $J_{k_1-1}(x) \ll x2^{-k_1}$ and $J_{k_2 -1 }(x) \ll x$ gives
\begin{equation}\label{eq: easy Q bound k}
    Q^*(m_1^2,c_1,m_2^2, c_2) \ll 2^{-k_1} (m_1m_2c_1c_2)^\epsilon m_1m_2 (c_1c_2)^{-1/2} R^{3\sigma/2}.
\end{equation}
Applying \eqref{eq: easy Q bound k} to \eqref{eq: Petersson} and setting $Y = (k_1k_2)^{9\epsilon}$ (so that the $Y^{-1/2+\epsilon}$ term appearing in \eqref{eq: Petersson} decays sufficiently quickly) gives
\begin{equation}
    \sum_{f \ot g \in H(k_1, N_1, k_2, N_2)} S(f \ot g; \phi) \ll 2^{-k_1} k_1k_2  R^{3\sigma/2} (k_1k_2)^\epsilon + R^{-\epsilon}\left(k_1k_2 + \dpole k_1R^{\sigma/2}\right).
\end{equation}
The proof of Theorem \ref{thm: over weight} follows after dividing by the size of the family using \eqref{eq: H size}.
\end{proof}

\section{Products of Kloosterman sums}\label{sec:kloosterman}
The remainder of the paper is dedicated to proving Theorem \ref{thm: extendo}. As such, we will assume that $N_1 = N_2 = N$ for the rest of the paper. In this section we analyze the Kloosterman sums arising from the Petersson formula. 
We are interested in the sum
\begin{align}\label{eq: Q extendo}
   Q^*(m_1^2, b_1N, m_2^2, b_2N) = \sum_p& S(m_1^2, p; b_1N) S(m_2^2, p; b_2N) J_{k_1 -1 } \left( \frac{4\pi m_1 \sqrt p}{b_1N }\right) J_{k_2 -1 } \left( \frac{4\pi m_2 \sqrt p}{b_2N } \right) \nn \\
   &\times \frac{2 \log p}{\sqrt p \log R} \hphi\left(\frac{\log p}{\log R} \right).
\end{align}
We will later use \eqref{eq: easy Q bound} to bound \eqref{eq: Q extendo} when $N$ divides $b_1$ or $b_2$, so for the rest of the section we assume that $(b_1, N) = (b_2, N) = 1$. Additionally, we have that $m_1, m_2 < N$, so we also assume $(m_1, N) = (m_2, N) = 1$. Our main result is the following.
\begin{prop}\label{lem: Qstar fail}
    Let $b_1, b_2, m_1, m_2$ be integers not divisible by the prime $N$. Set $r = (b_1, b_2)$ so that we can write $b_1 = d_1r$ and $b_2 = d_2 r$ with $(d_1, d_2) = 1$. We have that 
    \begin{align}\label{eq: Qstar fail}
        &Q^*(m_1^2, b_1N, m_2^2, b_2N)  \nn \\
        &\hs{1} = \frac{ 4\psi(m_1^2d_2^2, m_2^2 d_1^2, Nr)}{\varphi(d_1d_2rN)} R(m_1^2, d_1) R(m_2^2, d_2) \mu(d_1 d_2) \chi_0^{d_1d_2}(r) I(b_1, b_2, m_1, m_2, N) \nn \\
        &\hs{2}+   O\left(m_1^3 m_2^3   N^{2\sigma - 1/2 + \epsilon} (b_1b_2)^\epsilon \right)
    \end{align}
    where $R$ is the Ramanujan sum \eqref{def:RamSums}, $\chi_0^{d_1d_2}$ is the principal character modulo $d_1d_2$, $\psi$ is given in Lemma \ref{lem: psi lem}, and
    \begin{equation}\label{eq: I def}
        I(b_1, b_2, m_1, m_2, N) \clq   \int_0^\infty J_{k_1 -1}\left(\frac{4\pi m_1 y}{b_1N }\right) J_{k_2 -1}\left(\frac{4\pi m_2 y}{b_2N }\right) \hphi\left(2\frac{\log y}{\log R} \right) \frac{dy}{\log R}.
    \end{equation}
\end{prop}

To prove the proposition, we analyze the product of Kloosterman sums in \eqref{eq: Q extendo}. In Section \ref{sec: kloost} we decompose the Kloosterman sums in terms of Gauss sums in order to prove Lemma \ref{lem: Kloost N}. In Section \ref{sec: Dirichlet GRH} we apply GRH for Dirichlet $L$-functions in order to effectively bound our error terms. In Section \ref{sec: Gauss sum sum} we develop identities for sums of Gauss sums and Ramanujan sums in order to prove Lemma \ref{lem: psi lem}. Finally, in Section \ref{sec: Qstar eval} we apply partial summation to complete the proof of Proposition \ref{lem: Qstar fail}.
\subsection{Decomposing Kloosterman sums}\label{sec: kloost}
First, since $(b_1, N) = 1$, we can write
\begin{equation}\label{eq: Kloosterman split}
    S(m_1^2, p; b_1N) = S(\overline{N}m_1^2, \overline{N} p; b_1) S(\overline{b_1}m_1^2, \overline{b_1} p; N)
\end{equation}
where the overline denotes the multiplicative inverse modulo the period of the Kloosterman sum. The analogous result holds for $S(m_2^2, p; b_2N)$. We use the following lemma from \cite{ILS} for $S(\overline{N}m_1^2, \overline{N} p; b_1)$, and Lemma \ref{lem: Kloost N} for $S(\overline{b_1}m_1^2, \overline{b_1} p; N)$.

\begin{lemma}[\cite{ILS}, Section 6]\label{lem: Kloost b}
Let $p$ be a prime with $(p,b) = 1$ and let $(n, b) = 1$. Then
\begin{equation}
 S(nm, np ; b) = \frac{1}{\varphi(b)} \sum_{\chi (b)} \cbar(p) G_\chi(n^2m) G_\chi(1)
\end{equation}
where for a Dirichlet character $\chi$ modulo $b$,  $G_\chi$ is the Gauss sum defined in \eqref{eq:gausssum}.
\end{lemma}

If $p \nmid N$, we have that
\begin{equation}\label{eq: kloost double}
    S(m_1^2 \overline{b_1}, p \overline{b_1}; N)S(m_2^2 \overline{b_2}, p \overline{b_2}; N) = \frac{1}{\varphi(N)}\sum_{\chi (N)} \cbar(p) {\sum_{a (N)}}^* \chi(a)  S(m_1^2 \overline{b_1}, a  \overline{b_1}; N) S(m_2^2 \overline{b_2}, a \overline{b_2}; N).
\end{equation}
Since $\overline{b_1}$ and $\overline{b_2}$ are relatively prime to $N$, we have that
\begin{equation}\label{eq: kloost double K}
    S(m_1^2 \overline{b_1}, p \overline{b_1}; N)S(m_2^2 \overline{b_2}, p \overline{b_2}; N) = \frac{1}{\varphi(N)}\sum_{\chi (N)} \cbar(p) K(m_1^2 \overline{b_1}^2, m_2^2 \overline{b_2}^2, \chi),
\end{equation}
where
\begin{equation}\label{eq: K sum def}
    K(n_1, n_2, \chi) \clq {\sum_{a (N)}}^* \chi(a)  S(n_1, a; N)S(n_2, a; N).
\end{equation}
Recall that we assume that $N$ does not divide $b_1, b_2, m_1, m_2$ so that $N$ does not divide $m_1^2 \overline{b_1}^2$ and $m_2^2 \overline{b_2}^2$.
We need the following result.
\begin{lemma}\label{lem: Kloost N}
Let $N$ be a prime not dividing integers $n_1, n_2$,  $\chi$ a Dirichlet character modulo $N$, and $K(n_1, n_2, \chi)$ be as in \eqref{eq: K sum def}.
If $\chi = \chi_0$ is the principal character modulo $N$, then 
\begin{equation}
    K(n_1, n_2, \chi_0)  = \begin{cases}
    \varphi(N)^2 + \varphi(N) - 1 & n_1 - n_2 \equiv 0 (N) \\
    -\varphi(N)- 2 & \textrm{otherwise}. \end{cases}
\end{equation}
If $\chi$ is a non-principal Dirichlet character modulo $N$, we have that
\begin{equation}\label{eq: Kloost non prince}
    K(n_1, n_2, \chi) \ll N^{3/2}.
\end{equation}
\end{lemma}
\begin{rem}
When $\chi_0$ is principal and $n_1 - n_2 \equiv 0 (N)$, we have that $S(n_1, a; N) = S(n_2, a;N)$ so that all the terms in \eqref{eq: K sum def} are positive (Kloosterman sums are real numbers for any integer arguments). In the other cases when $\chi_0$ is principal, the sign of the product of Kloosterman sums changes, which leads to better than square root cancellation. When $\chi$ is non-principal, the sum \eqref{eq: Kloost non prince} exhibits square root cancellation.
\end{rem}
\begin{proof}

Expanding the Kloosterman sums and rearranging gives
\begin{align}
    K(n_1, n_2, \chi)
    &= {\sum_{a (N)}}^* \chi(a) {\sum_{u_1 (N)}}^* e\left(\frac{au_1 + n_1\overline{u_1}}{N} \right) {\sum_{u_2 (N)}}^* e\left(\frac{au_2 + n_2 \overline{u_2}}{N} \right)\nn \\
    &= {\sum_{u_1 (N)}}^*  {\sum_{u_2 (N)}}^* e\left(\frac{n_1 \overline{u_1} + n_2 \overline{u_2}}{N} \right) G_\chi(u_1 + u_2).
\end{align}
Since $N$ is prime, we have that
\begin{equation}\label{eq: Gauss sum cases}
    G_\chi(u_1 + u_2) = \begin{cases} \delta_\chi \varphi(N) & u_1 + u_2 \equiv 0 (N) \\
    \cbar(u_1 + u_2)G_\chi(1) & \textrm{otherwise} \end{cases}
\end{equation}
so we can write
\begin{align}\label{eq: K first Gauss}
    K(n_1, n_2, \chi)    &= G_\chi(1) {\sum_{u_1 (N)}}^*  {\sum_{u_2 (N)}}^* e\left(\frac{n_1 \overline{u_1} + n_2\overline{u_2}}{N} \right) \cbar(u_1 + u_2)\nn \\ 
     &\hs{1} + \delta_\chi \varphi(N) {\sum_{u_1 (N)}}^* e\left( \frac{\overline{u_1} (n_1- n_2)}{N} \right)
\end{align}
where $\delta_\chi$ is the indicator function for the principal character. This second sum equals $\varphi(N)$ when $n_1 - n_2 \equiv 0 \mod N$ and is $\mu(N)$ otherwise. For the first sum, we do a change of variables $u_1 \to u_1u_2$ which gives
\begin{align}\label{eq: u1 u2 sum}
    {\sum_{u_1 (N)}}^*  {\sum_{u_2 (N)}}^* e\left(\frac{n_1 \overline{u_1} + n_2\overline{u_2}}{N} \right) \cbar(u_1 + u_2)  &= {\sum_{u_1 (N)}}^*  {\sum_{u_2 (N)}}^* e\left(\frac{n_1 \overline{u_1u_2} + n_2 \overline{u_2}}{N} \right) \cbar(u_1u_2 + u_2) \nn \\
    &= {\sum_{u_1 (N)}}^* \cbar(u_1 + 1) {\sum_{u_2 (N)}}^* e\left(\frac{\overline{u_2} (n_1 \overline{u_1} + n_2)}{N} \right) \chi(\overline{u_2}) \nn \\
    &= {\sum_{u_1 (N)}}^* \cbar(u_1 + 1) G_\chi(n_1 \overline{u_1} + n_2).
\end{align}
Applying \eqref{eq: Gauss sum cases} to \eqref{eq: u1 u2 sum} gives
\begin{equation}
     {\sum_{u_1 (N)}}^*  {\sum_{u_2 (N)}}^* e\left(\frac{n_1 \overline{u_1} + n_2\overline{u_2}}{N} \right) \cbar(u_1 + u_2) = G_\chi(1) {\sum_{u_1 (N)}}^* \cbar(u_1 + 1) \cbar (n_1 \overline{u_1} + n_2) + \chi_0(-n_1\overline{n_2}+ 1) \delta_\chi \varphi(N).
\end{equation}
Applying this to \eqref{eq: K first Gauss} gives
\begin{equation}
    K(n_1, n_2, \chi)    = G_\chi(1)^2 {\sum_{u_1 (N)}}^* \cbar(u_1 + 1) \cbar (n_1 \overline{u_1} + n_2) + \delta_\chi D(n_1, n_2)
\end{equation}
where
\begin{equation}
    D(n_1, n_2) = \begin{cases}
    \varphi(N)^2 & n_1 - n_2 \equiv 0 (N) \\
    -2\varphi(N) & \textrm{otherwise.} \end{cases}
\end{equation}
Thus when $\chi$ is principal we have
\begin{equation}
    K(n_1, n_2, \chi)  = \begin{cases}
    \varphi(N)^2 + \varphi(N) - 1 & n_1 - n_2 \equiv 0 (N) \\
    -\varphi(N)- 2 & \textrm{otherwise} \end{cases}
\end{equation}
as desired.
When $\chi$ is non-principal, we have that
\begin{equation}
    {\sum_{u_1 (N)}}^* \cbar(u_1 + 1) \cbar (n_1 \overline{u_1} + n_2) \ll N^{1/2}.
\end{equation}
This follows from Weil's bound on character sums; see \cite{IK} equation (12.23). The proof follows from the fact that $|G_\chi(1)| = \sqrt{N}$ when $\chi$ is primitive modulo $N$.
\end{proof}

\subsection{Applying GRH for Dirichlet $L$-functions}\label{sec: Dirichlet GRH}

We study a modified version of \eqref{eq: Q extendo} defined as
\begin{equation}\label{eq: double kloost sum}
    A\clq A(x, m_1^2, m_2^2, b_1, b_2, N) \clq \sum_{p \le x} S(m_1^2, p ; b_1N) S(m_2^2, p ;b_2N) \log p
\end{equation}
and then derive a closed form for $Q^*$ using partial summation. We study this sum using GRH for Dirichlet $L$-functions. This implies for a Dirichlet character $\chi$ modulo $c$ that
\begin{equation}\label{eq: GRH Dirichlet}
    \sum_{p \le x} \chi(p) \log p = \delta_{\chi} x + O\left(x^{1/2}(cx)^\epsilon \right)
\end{equation}
where $\delta_\chi$ is the indicator for the principal character. Applying \eqref{eq: Kloosterman split}, \eqref{eq: kloost double K} and Lemma \ref{lem: Kloost b} gives
\begin{align}\label{eq: A expanded}
    A &= \frac{1}{\varphi(b_1)\varphi(b_2)\varphi(N)}\sum_{\chi_1 (b_1)} \sum_{\chi_2 (b_2)} \sum_{\chi_3 (N)} G_{\chi_1}(m_1^2\overline{N}^2)G_{\chi_1}(1) G_{\chi_2}(m_2^2\overline{N}^2)G_{\chi_2}(1) K(m_1^2 \overline{b_1}^2, m_2^2 \overline{b_2}^2, \chi_3) \nn \\
    &\hs{1} \times \sum_{p\le x} \overline{\chi_1\chi_2\chi_3}(p) \log p .
\end{align}
Note that we do not account for when $p | b_1N$ or $p|b_2N$, but these terms are absorbed by the error term \eqref{eq: kloost error term}. The main term of $A$ is when $\overline{\chi_1\chi_2\chi_3}$ is principal. This occurs when $\chi_3$ is principal and $\chi_1$ is induced by some character $\chi^*$ modulo $(b_1, b_2)$, and $\chi_2$ is induced by $\cbar^*$. Using Lemma \ref{lem: Kloost N}, we can write the main term of $A$ as 
\begin{equation}
    \frac{ \psi(m_1^2 b_2^2, m_2^2b_1^2, N)x}{\varphi(b_1)\varphi(b_2)\varphi(N)} \sum_{\substack{\chi (b_1, b_2)}} G_{\chi_1}(m_1^2\overline{N}^2)G_{\chi_1}(1) G_{\chi_2}(m_2^2\overline{N}^2)G_{\chi_2}(1)
\end{equation}
where $\chi_1$ is the character modulo $b_1$ induced by $\chi$, $\chi_2$ is the character modulo $b_2$ induced by $\overline{\chi}$, and
\begin{equation}\label{eq: psi first def}
\psi(n_1, n_2, N) \clq
    \begin{cases}
    \varphi(N)^2 + \varphi(N) - 1 & n_1 - n_2 \equiv 0 (N) \\
    -\varphi(N)- 2 & \textrm{otherwise}. \end{cases}
\end{equation}
We have that
\begin{equation}
    G_{\chi_1}(m_1^2\overline{N}^2) = \chi_1(N)^2 G_{\chi_1}(m_1^2)
\end{equation}
so we can simplify the main term as
\begin{equation}
    \frac{ \psi(m_1^2 b_2^2, m_2^2b_1^2, N)x}{\varphi(b_1)\varphi(b_2)\varphi(N)} \sum_{\substack{\chi (b_1, b_2)}} G_{\chi_1}(m_1^2)G_{\chi_1}(1) G_{\chi_2}(m_2^2)G_{\chi_2}(1)
\end{equation}
since $\chi_1(N) \chi_2(N) = \chi_0^{b_1}(N)\chi_0^{b_2}(N)\chi(N)\cbar(N) = 1$, where $\chi_0^n$ denotes the principal character modulo $n$.

Applying \eqref{eq: Gauss sum bound MV}, \eqref{eq: GRH Dirichlet} and Lemma \ref{lem: Kloost N}, the error term in \eqref{eq: A expanded} can be bounded by
\begin{equation}\label{eq: kloost error term}
    x^{1/2}b_1b_2m_1^2 m_2^2 N^{3/2}(b_1b_2Nx)^\epsilon.
\end{equation}
This gives the following expression for $A$:
\begin{align}\label{eq: A Gauss}
    A &= \frac{ \psi(m_1^2 b_2^2, m_2^2b_1^2, N)x}{\varphi(b_1)\varphi(b_2)\varphi(N)} \sum_{\substack{\chi (b_1, b_2)}} G_{\chi_1}(m_1^2)G_{\chi_1}(1) G_{\chi_2}(m_2^2)G_{\chi_2}(1) \nn \\
    &\hs{1} + O\left( x^{1/2}b_1b_2 m_1^2 m_2^2  N^{3/2}(b_1b_2Nx)^\epsilon \right).
\end{align}

\subsection{Sums of Gauss sums}\label{sec: Gauss sum sum}

We want to analyze the sum over Gauss sums in \eqref{eq: A Gauss}. We begin by reducing the induced characters $\chi_1$ and $\chi_2$ to the character $\chi$ modulo $(b_1, b_2)$. Before we begin, we need to introduce some notation. Set $r = (b_1, b_2)$, $r_1 = (b_1, r^\infty)$, and $r_2 = (b_2, r^\infty)$. We first prove the following lemma.
\begin{lemma}\label{lem: single gauss sum}
    Let $b_1, r, r_1, \chi, \chi_1$ be as above. We have that
    \begin{equation}
        G_{\chi_1}(m_1^2) = \begin{cases}
            \chi(b_1/r_1) R(m_1^2, b_1/r_1) \frac{r_1}{r} G_{\chi} \left(m_1^2 r/r_1\right) & r_1 | m_1^2 r \\
            0 & \textrm{otherwise}.
        \end{cases}
    \end{equation}
\end{lemma}
\begin{proof}
We have that $(r_1, b_1/r_1) = 1$, so we can write $\chi_1 = \chi_0^{b_1/r_1} \chi$. We then have that
\begin{align}\label{eq: Gauss sum decomp}
    G_{\chi_1}(m_1^2) &= \sum_{u (b_1)} \chi_1(u) e\left( \frac{um_1^2}{b_1}\right) \nn \\
    &= \sum_{u_1 (r_1)} \sum_{u_2 (b_1/r_1)} \chi_1(u_1 b_1/r_1 + u_2r_1) e\left( \frac{u_1m_1^2}{r_1}\right) e\left( \frac{u_2m_1^2}{b_1/r_1}\right) \nn \\
    &= \chi(b_1/r_1) \chi_0^{b_1/r_1}(r_1) \sum_{u_1 (r_1)} \chi(u_1) e\left( \frac{u_1m_1^2}{r_1}\right)
    \sum_{u_2 (b_1/r_1)} \chi_0^{b_1/r_1}(u_2) e\left( \frac{u_2m_1^2}{b_1/r_1}\right) \nn \\
    &= \chi(b_1/r_1) R(m_1^2, b_1/r_1) \sum_{u_1 (r_1)} \chi(u_1) e\left( \frac{u_1m_1^2}{r_1}\right).
\end{align}
Now, we have that $r | r_1$, so we can write $u_1 = u_3 + ru_4$ with $u_3$ going from 1 to $r$ and $u_4$ going from 1 to $r_1/r$, so that
\begin{align}
    \sum_{u_1 (r_1)} \chi(u_1) e\left( \frac{u_1m_1^2}{r_1}\right) &= \sum_{u_3(r)} \sum_{u_4 (r_1/r)} \chi(u_3 + u_4 r) e\left( \frac{(u_3 +u_4r) m_1^2}{r_1}\right) \nn \\
    &= G_\chi\left( \frac{m_1^2 r}{r_1} \right) \sum_{u_4(r_1/r)} e\left( \frac{u_4r m_1^2}{r_1}\right).
\end{align}
This final sum equals $r_1/r$ if $r_1 | rm_1^2$ and is 0 otherwise. Substituting this back into \eqref{eq: Gauss sum decomp} completes the proof.
\end{proof}
Now, if $m_1 = 1$, we have that $G_{\chi_1}(1)$ is 0 unless $r = r_1$, so that $(r, b_1/r) = 1$. In this case, we also have that $R(1, b_1/r_1) = \mu(b_1/r)$, so that
\begin{equation}\label{eq: Gauss sum 1}
    G_{\chi_1}(1) = \chi(b_1/r) \mu(b_1/r) G_\chi(1).
\end{equation}
But if $r = r_1$, then $r_1 | m_1^2r$, so that
\begin{equation}\label{eq: Gauss sum m}
        G_{\chi_1}(m_1^2) = \chi(b_1/r) R(m_1^2, b_1/r)  G_{\chi} \left(m_1^2\right) 
    \end{equation}
by Lemma \ref{lem: single gauss sum}.
Of course, the analogs of \eqref{eq: Gauss sum 1} and \eqref{eq: Gauss sum m} hold for $\chi_2$, which was induced from $\cbar$. Set $d_1 = b_1/r$ and $d_2 = b_2/r$ so that $(d_1, d_2) = 1$. Applying \eqref{eq: Gauss sum 1} and \eqref{eq: Gauss sum m} gives
\begin{align}\label{eq: character drop}
    \sum_{\substack{\chi (r)}} G_{\chi_1}(m_1^2)G_{\chi_1}(1) G_{\chi_2}(m_2^2)G_{\chi_2}(1) &= R(m_1^2, d_1) R(m_2^2, d_2) \mu(d_1d_2)\chi_0^{r}(d_1d_2)\nn \\
    &\hs{1} \times \sum_{\chi (r)} \chi(d_1^2\overline{d_2}^2) G_\chi(m_1^2) G_{\chi}(1)  G_{\cbar}(m_2^2) G_{\cbar}(1).
\end{align}
Because of the term $\chi_0^r(d_1d_2)$, we may assume that $(r,d_1) = (r,d_2) = 1$.
Now, we have
\begin{align}
    G_\chi(m_1^2)G_{\cbar}(1) &= \sum_{u_1(r)} \chi(u_1) e\left(\frac{u_1m_1^2}{r}\right) \sum_{u_2(r)} \cbar(u_2)e\left(\frac{u_2}{r}\right) \nn\\
    &= \sum_{u_1 (r)} \chi(u_1) e\left(\frac{u_1m_1^2}{r}\right) \sum_{u_2(r)} \cbar(u_1u_2) e\left(\frac{u_1u_2}{r}\right)\nn  \\
    &= \sum_{u_2(r)} \cbar(u_2) {\sum_{u_1(r)} }^* e\left(\frac{u_1(u_2+m_1^2)}{r}\right) \nn \\
    &= \sum_{u_2(r)} \cbar(u_2) R(u_2+m_1^2, r).
\end{align}
Applying this gives
\begin{align}
    \sum_{\chi (r)} \chi(d_1^2\overline{d_2}^2) G_\chi(m_1^2) G_{\chi}(1)  G_{\cbar}(m_2^2) G_{\cbar}(1) &= {\sum_{u_1(r)}}^*  R(u_1+m_1^2, r) {\sum_{u_2(r)}}^*  R(u_2+m_2^2, r) \sum_{\chi (r)} \chi(\overline{u_1 } u_2  d_1^2 \overline{d_2^2}).
\end{align}
By orthogonality the inner sum equals 0 unless $u_2 = u_1\overline{d_1^2}d_2^2$, in which case it is $\varphi(r)$. Thus we have that
\begin{align}\label{eq: Gauss to Ram}
    \sum_{\chi (r)} \chi(d_1^2\overline{d_2}^2) G_\chi(m_1^2) G_{\chi}(1)  G_{\cbar}(m_2^2) G_{\cbar}(1) &= \varphi(r) {\sum_{u_1 (r)}}^* R(u_1+m_1^2, r) R(u_1 \overline{d_1^2}d_2^2 + m_2^2, r) \nn \\
    &=\varphi(r) {\sum_{u_1 (r)}}^* R(u_1 + m_1^2 d_2^2, r) R(u_1 +  m_2^2 d_1^2, r) .
\end{align}
Now, we want to study sums of the type
\begin{equation}\label{eq: psi def}
    \psi(n_1, n_2, r) \clq {\sum_{u_1(r)}}^* R(u_1 +n_1, r)R(u_1 +n_2, r).
\end{equation}
We obtain the following result.
\begin{lemma}\label{lem: psi lem}
    The function $\psi(n_1, n_2, r)$ defined in \eqref{eq: psi def} is multiplicative in $r$ so it can be defined by its values when $r = p^\alpha$. We have that
    \begin{equation}
        \psi(n_1, n_2, p^\alpha) = \begin{cases}
            pR(n_1 - n_2, p) - R(n_1, p) R(n_2, p) & \alpha = 1 \\
            0 & \alpha > 1 \normalfont{\textrm{ and } }p|n_1n_2 \\
            p^\alpha R(n_1 - n_2, p^\alpha) & \normalfont{\textrm{otherwise.}}
        \end{cases}
    \end{equation}
\end{lemma}
\begin{rem}
    The function $\psi(n_1, n_2, r)$ defined in \eqref{eq: psi def} is equivalent to the function $\psi(n_1, n_2, N)$ defined in \eqref{eq: psi first def} when $r$ is prime. One can easily verify that the definitions agree using Lemma \ref{lem: psi lem}.
\end{rem}
\begin{proof}
    First we show that $\psi$ is multiplicative. Write $r= st$ with $(s,t) = 1$. As the Ramanujan sums are multiplicative, $R(a,r) = R(a,s)R(a,t)$. Writing $u_1 = u_2s+u_3t$ in the sum gives
    \begin{align}
        \psi(n_1, n_2, r) 
        &= {\sum_{u_2(t)}}^* {\sum_{u_3(s)}}^* R(u_2s + u_3t + n_1, s) R(u_2s + u_3t + n_1, t)\nn \\
        &\hspace{1.7cm} \times R(u_2s + u_3t +n_2, s) R(u_2s + u_3t + n_2, t) \nn \\
        &= {\sum_{u_2(t)}}^* R(u_2s + n_1, t) R(u_2s + n_2, t) {\sum_{u_3(s)}}^* R(u_3t + n_1, s) R(u_3t + n_2, s) 
    \end{align}
    since $R(a, r)$ is periodic modulo $r$. Doing a change of variables $u_2 \to u_2\overline{s}$ and $u_3 \to u_3 \overline{t}$ gives 
    \begin{align}
        \psi(n_1, n_2,r) 
        &= \psi(n_1, n_2, s)\psi(n_1, n_2,t)
    \end{align}
    as desired. 

    Now we evaluate $R(d,r)$ when $r = p^\alpha$ with $\alpha \ge 1$. We can write
    \begin{equation}
        \psi(n_1, n_2, p^\alpha) = \sum_{u_1 (p^\alpha)} R(u_1 + n_1, p^\alpha) R(u_1 + n_2, p^\alpha) - \sum_{u_1 (p^{\alpha-1})} R(u_1p + n_1, p^\alpha) R(u_1p + n_2, p^\alpha).
    \end{equation}
    Call the first sum $S_1$ and the second $S_2$. We have that
    \begin{align}
        S_1 &= {\sum_{u_1 (p^\alpha)}} {\sum_{u_2 (p^\alpha)}}^* e \left( \frac{u_1u_2 + n_1u_2}{p^\alpha}\right) {\sum_{u_3 (p^\alpha)}}^* \left( \frac{u_1u_3 + n_2u_3}{p^\alpha}\right) \nn \\
        &= {\sum_{u_2 (p^\alpha)}}^* {\sum_{u_3 (p^\alpha)}}^* e \left( \frac{n_1 u_2 + n_2 u_3}{p^\alpha}\right) {\sum_{u_1 (p^\alpha)}} e \left( \frac{u_1 (u_2 + u_3)}{p^\alpha}\right).
    \end{align}
    The inner sum is 0 unless $u_2 + u_3 \equiv 0 (p^\alpha)$, so we have that
    \begin{align}
        S_1 &= p^\alpha {\sum_{u_2 (p^\alpha)}}^* e \left( \frac{u_2 (n_1 - n_2)}{p^\alpha}\right) = p^\alpha R(n_1 - n_2, p^\alpha).
    \end{align}
    If $\alpha = 1$, we have that $S_2 = R(n_1, p)R(n_2, p)$. If $\alpha > 1$, we have
    \begin{align}
        S_2 &= {\sum_{u_2 (p^\alpha)}}^* {\sum_{u_3 (p^\alpha)}}^* e \left( \frac{n_1 u_2 + n_2 u_3}{p^\alpha}\right) {\sum_{u_1 (p^{\alpha-1})}} e \left( \frac{u_1 (u_2 + u_3)}{p^{\alpha - 1}}\right) .
    \end{align}
    The inner sum is 0 unless $u_2 + u_3 \equiv 0 (p^{\alpha - 1})$. Thus we can write $u_3 = -u_2 + u_4p^{\alpha - 1}$ so that
    \begin{align}
        S_2 &= p^{\alpha - 1} {\sum_{u_2 (p^\alpha)}}^* e \left( \frac{u_2 (n_1 - n_2)}{p^{\alpha }}\right) \sum_{u_4 (p)} e \left( \frac{u_4n_2}{p}\right) = p^\alpha R(n_1 - n_2, p^\alpha)
    \end{align}
    if $p| n_2$, and $S_2 = 0$ otherwise. Taking $S_1 - S_2$ completes the lemma.
\end{proof}

Now, applying \eqref{eq: psi def} to \eqref{eq: Gauss to Ram} and then plugging into \eqref{eq: character drop} gives
\begin{equation}
    \sum_{\substack{\chi (r)}} G_{\chi_1}(m_1^2)G_{\chi_1}(1) G_{\chi_2}(m_2^2)G_{\chi_2}(1) = 
    R(m_1^2, d_1) R(m_2^2, d_2) \mu(d_1d_2)\chi_0^{r}(d_1d_2)\varphi(r) \psi( m_1^2 d_2^2,  m_2^2 d_1^2, r).
\end{equation}
Applying this to \eqref{eq: A Gauss} and using the identity $\psi(m_1^2 b_2^2, m_2^2 b_1^2, N) = \psi(m_1^2 d_2^2, m_2^2 d_1^2, N)$ we finally have
\begin{align}\label{eq: A final}
    A &= \frac{ \psi(m_1^2 d_2^2, m_2^2 d_1^2, Nr)}{\varphi(d_1 d_2 Nr) } R(m_1^2, d_1) R(m_2^2, d_2) \mu(d_1 d_2) \chi_0^{d_1d_2}(r)     x \nn \\
    &\hs{1} + O\left( x^{1/2}b_1b_2m_1^2 m_2^2 N^{3/2}(b_1b_2Nx)^\epsilon \right).
\end{align}

\subsection{Evaluating $Q^*$}\label{sec: Qstar eval}

We use summation by parts to express \eqref{eq: Q extendo} in terms of \eqref{eq: A final}. Doing so gives 
\begin{align}
    Q^* &= -\int_0^\infty \left[ \frac{ \psi(m_1^2 d_2^2, m_2^2 d_1^2, Nr)}{\varphi(d_1 d_2 Nr) } R(m_1^2, d_1) R(m_2^2, d_2) \mu(d_1 d_2) \chi_0^{d_1d_2}(r)     x + O\left( x^{1/2}b_1b_2m_1^2 m_2^2 N^{3/2}(b_1b_2Nx)^\epsilon \right)\right]  \nn\\
    & \hs{1} \times d J_{k_1 -1}\left(\frac{4\pi m_1 \sqrt x}{b_1N }\right) J_{k_2 -1}\left(\frac{4\pi m_2 \sqrt x}{b_2N }\right) \frac{ 2}{\sqrt x \log R} \hphi\left(\frac{\log x}{\log R} \right).
\end{align}
Integrating by parts and setting $y = \sqrt x$ gives that the main term is
\begin{align}
     &\frac{ 4\psi(m_1^2d_2^2, m_2^2 d_1^2, Nr)}{\varphi(d_1d_2rN)} R(m_1^2, d_1) R(m_2^2, d_2) \mu(d_1 d_2) \chi_0^{d_1d_2}(r)\nn \\
     &\hs{1} \times \int_0^\infty J_{k_1 -1}\left(\frac{4\pi m_1 y}{b_1N }\right) J_{k_2 -1}\left(\frac{4\pi m_2 y}{b_2N }\right) \hphi\left(2\frac{\log y}{\log R} \right) \frac{dy}{\log R}.  
\end{align}
Similarly, we can bound the error term by
\begin{equation}
    O\left(b_1b_2m_1^2 m_2^2 N^{3/2}(b_1b_2N)^\epsilon\right) \int_0^{R^\sigma}  \left|J_{k_1 -1}\left(\frac{4\pi m_1 \sqrt x}{b_1N }\right) J_{k_2 -1}\left(\frac{4\pi m_2 \sqrt x}{b_2N }\right)\right| \frac{dx}{x}
\end{equation}
and using $J_{k-1}(x) \ll x$ gives that this is bounded by 
\begin{equation}
    m_1^3m_2^3 N^{-1/2}R^{\sigma}(b_1b_2N)^\epsilon \ll m_1^3m_2^3 N^{2\sigma - 1/2 + \epsilon} (b_1b_2)^\epsilon.
\end{equation}
Putting this together gives Proposition \ref{lem: Qstar fail}.\qed

\section{Surpassing (-1, 1): proof of Theorem 1.2}\label{sec:extendo}
In this section, we complete the proof of Theorem \ref{thm: extendo} in the case where $k_1 \ne k_2$ by proving the following proposition.

\begin{prop}\label{prop: 6 god}
Let $k_1 \ne k_2$ and set
    \begin{align}\label{eq: PkkN def}
    \mathcal{P}(k_1, k_2, N) &\clq \frac{1}{|H(k_1, N, k_2, N)|} \sum_{f,g} S(f \ot g; \phi).
\end{align}
If $\supp \hphi  \subset (-5/4, 5/4)$, we have that
\begin{equation}
    \lim_{N\to \infty} \mathcal{P}(k_1, k_2, N) = \infint \phi(x) \frac{\sin(2\pi x)}{2\pi x} dx - \frac{1}{2}\phi(0).
\end{equation}
\end{prop}
Combining Proposition \ref{prop: 6 god} with Proposition \ref{prop: explicit formula} completes the proof of Theorem \ref{thm: extendo} (after comparing with \eqref{eq: W def}) in the case where $k_1 \ne k_2$, as in this case there is no polar contribution. The key insight which allows us to obtain a closed form for $\mathcal{P}(k_1, k_2, N)$ in the limit as $N \to \infty$ is Lemma \ref{lem: full restriction}, in which we apply Proposition \ref{lem: Qstar fail}. In doing so, we are able to remove many lower order subterms, and the integral which remains involves a product of Bessel functions with a relatively simple Mellin transform. We evaluate this integral in Section \ref{sec: IrN eval} using methods similar to Section 7 of \cite{ILS}.

\subsection{Removing subterms}
In this subsection, assume that $k_1 \ne k_2$.
Applying Proposition \ref{prop: Petersson final} to \eqref{eq: PkkN def} with $Y = N^{8\epsilon}$ gives
\begin{equation}\label{eq: PkkN Petersson}
    \mathcal{P}(k_1, k_2, N) = \frac{4\pi^2i^{k_1 + k_2}}{\varphi(N)^2} \sum_{m_1, m_2 \le Y} \frac{1}{m_1m_2} \sum_{b_1, b_2  \ge 1} \frac{1}{b_1b_2} Q^*(m_1^2,b_1N,m_2^2, b_2N) + \ONE.
\end{equation}
We begin by using \eqref{eq: easy Q bound} to bound terms where $b_1$ and $b_2$ are large so that we may apply Proposition \ref{lem: Qstar fail}.
\begin{lemma}\label{lemma: small b}
    If $\supp (\hphi) \subset (-2, 2)$, we have that
    \begin{equation}\label{eq: PkkN small b}
        \mathcal{P}(k_1, k_2, N) = \frac{4\pi^2i^{k_1 + k_2}}{\varphi(N)^2} \sum_{m_1, m_2 \le Y} \frac{1}{m_1m_2} \sum_{1 \le b_1, b_2  < N^6} \frac{1}{b_1b_2} Q^*(m_1^2,b_1N,m_2^2, b_2N) + \ONE.
    \end{equation}
\end{lemma}
\begin{rem}
    We restrict the size of $b_1, b_2$ so that sums over $b_1\inv$ and $b_2\inv$ converge and are small. The restriction to $N^6$ in \eqref{eq: PkkN small b} is arbitrary. It could instead be changed to any sufficiently large power of $N$, which would allow for greater support of $\hphi$.
\end{rem}
\begin{proof}
    By \eqref{eq: easy Q bound} and the fact that $m_1, m_2 \le Y = N^{8\epsilon}$ we have that
    \begin{align}\label{eq: Q bound simple}
        Q^*(m_1^2,b_1N,m_2^2, b_2N) &\ll (m_1m_2b_1Nb_2N)^\epsilon m_1m_2 (b_1Nb_2N)^{-1/2} N^{3\sigma} \nn \\
        &\ll N^{3\sigma + \epsilon - 1}(b_1b_2)^{-1/2 + \epsilon}.
    \end{align}
    We can bound the terms with $b_1 \ge N^6$ or $b_2 \ge N^6$ in \eqref{eq: PkkN Petersson} by
    \begin{equation}
        N^{-2} N^\epsilon N^{3\sigma + \epsilon - 1} \sum_{b_1 \ge N^6} \frac{1}{b_1^{3/2 - \epsilon}} \sum_{b_2 \ge 1} \frac{1}{b_2^{3/2 - \epsilon}} \ll N^{3\sigma + \epsilon - 6}.
    \end{equation}
    This is $\ONE$ if $\sigma < 2$.
\end{proof}
Next we remove terms where $N$ divides $b_1$ or $b_2$.
\begin{lemma}\label{lemma: b divides N}
    If $\supp (\hphi) \subset (-3/2, 3/2)$, we have that
    \begin{equation}\label{eq: PkkN b divides N}
        \mathcal{P}(k_1, k_2, N) = \frac{4\pi^2i^{k_1 + k_2}}{\varphi(N)^2} \sum_{m_1, m_2 \le Y} \frac{1}{m_1m_2} \sum_{\substack{1 \le b_1, b_2  < N^6 \\ N \nmid b_1b_2}} \frac{1}{b_1b_2} Q^*(m_1^2,b_1N,m_2^2, b_2N) + \ONE.
    \end{equation}
\end{lemma}
\begin{proof}
We need to bound the terms in \eqref{eq: PkkN small b} with $N|b_1$ or $N | b_2$.
Using \eqref{eq: easy Q bound} we find that we can bound these terms by
    \begin{equation}
        N^{-2} N^\epsilon N^{3\sigma + \epsilon - 1} \sum_{c_1 \ge 1} \frac{1}{(Nc_1)^{3/2 - \epsilon}} \sum_{b_2 \ge 1} \frac{1}{b_2^{3/2 - \epsilon}} \ll N^{3\sigma + \epsilon - 9/2}.
    \end{equation}
    This is $\ONE$ if $\sigma < 3/2$.
\end{proof}

We are now ready to remove additional terms and simplify using Proposition \ref{lem: Qstar fail}. The remaining terms are those where $m_1$ is a multiple of $d_1$ and $m_2$ is a multiple of $d_2$.
\begin{lemma}\label{lem: full restriction}
    If $\supp (\hphi) \subset (-5/4, 5/4)$, we have that
    \begin{align}\label{eq: P post cancellation}
        \mathcal{P}(k_1, k_2, N) &= \frac{16\pi^2i^{k_1 + k_2}\psi(1,1,N)}{\varphi(N)^3} \sum_{m \le Y} \frac{1}{m^2} \sum_{\substack{d_1, d_2 \le Y/m }} \frac{\mu(d_1d_2)}{d_1^2d_2^2} \sum_{\substack{(r, d_1d_2) = 1 }} \frac{\psi(m^2, m^2, r)}{r^2 \varphi(r)} I(m, r, N)\nn \\
        &\hs{1} + \ONE
    \end{align}
    where
    \begin{equation}\label{eq: IrN def}
        I(m, r, N) \clq   \int_0^\infty J_{k_1 -1}\left(\frac{4\pi m y}{rN }\right) J_{k_2 -1}\left(\frac{4\pi m y}{rN }\right) \hphi\left(2\frac{\log y}{\log R} \right) \frac{dy}{\log R}.
    \end{equation}
\end{lemma}
\begin{proof}
    First we want a general purpose bound for $Q^*$ using Proposition \ref{lem: Qstar fail}. We have that $\varphi(n) \gg n^{1- \epsilon}$, $\psi(n_1, n_2, r) \ll r^2$, $R(n, d) \le \varphi(d)$ and using $J_{\nu}(x) \ll x$, we can bound the integral piece by
\begin{equation}
    \frac{m_1m_2}{b_1b_2N^{2}}\int_0^{R^{\sigma/2}} y^2 dy \ll m_1m_2 (b_1b_2)^{-1} N^{3\sigma  - 2}.
\end{equation}
This gives the bound
\begin{equation}\label{eq: Qstar bound not equal}
    Q^*(m_1^2,b_1N,m_2^2, b_2N) \ll \psi(m_1^2d_2^2, m_2^2d_1^2, N) r^{-1+\epsilon}(d_1d_2)^{-1} m_1m_2 N^{3\sigma + \epsilon - 3} + m_1^3m_2^3 N^{2\sigma - 1/2 + \epsilon} (b_1b_2)^\epsilon.
\end{equation}
Now, if $m_1^2d_2^2 - m_2^2d_1^2 \not\equiv 0 (N)$, we have that $\psi(m_1^2d_2^2, m_2^2d_1^2, N) \ll N$, so we have the bound
    \begin{equation}
        Q^*(m_1^2,b_1N,m_2^2, b_2N) \ll r^{-1+\epsilon}(d_1d_2)^{-1}m_1m_2 N^{3\sigma + \epsilon - 2} + m_1^3 m_2^3 N^{2\sigma - 1/2 + \epsilon} (b_1b_2)^\epsilon.
    \end{equation}
    Using this bound and $m_1, m_2 \le Y = N^{8\epsilon}$, we find that we can bound the terms in  \eqref{eq: PkkN b divides N} with $m_1^2d_2^2 - m_2^2d_1^2 \not\equiv 0 (N)$ by
    \begin{equation}
        N^{-2 + \epsilon}  \sum_{1 \le r, d_1, d_2 < N^6} \left[\frac{N^{3\sigma + \epsilon - 2}}{r^{3 - \epsilon}d_1^2d_2^2} + \frac{N^{2\sigma + \epsilon - 1/2}}{r^{2-\epsilon}d_1^{1- \epsilon}d_2^{1-\epsilon}}   \right] \ll N^{3\sigma + \epsilon - 4} + N^{2\sigma - 5/2 + \epsilon}.
    \end{equation}
    This is $\ONE$ if $\supp (\hphi) \subset (-5/4, 5/4)$. 

    For the remaining terms, we have that $m_1^2d_2^2 - m_2^2d_1^2 \equiv 0 (N)$, so 
    \begin{equation}
        Q^*(m_1^2,b_1N,m_2^2, b_2N) \ll r^{-1+\epsilon}(d_1d_2)^{-1} m_1m_2 N^{3\sigma + \epsilon - 1} + m_1^3 m_2^3 N^{2\sigma - 1/2 + \epsilon} (b_1b_2)^\epsilon.
    \end{equation}
    If $m_1^2d_2^2 - m_2^2d_1^2 \equiv 0 (N)$, we have that $m_1d_2 \equiv \pm m_2 d_1 (N)$. If $m_1d_2 \ne m_2d_1$, we have that either $m_1d_2 > N/2$ or $m_2d_1 > N/2$. Since $m_1, m_2 \ll N^{8\epsilon}$ this means either $d_1 \gg N^{1 - \epsilon}$ or $d_2 \gg N^{1 - \epsilon}$. 
    Thus we can bound the terms in \eqref{eq: PkkN b divides N} with $m_1^2d_2^2 - m_2^2d_1^2 \equiv 0 (N)$ and $m_1d_2 \ne m_2d_1$ by

    \begin{equation}
        N^{-2 + \epsilon}  \sum_{1 \le r, d_2 < N^6} \sum_{N^{1 - \epsilon} \ll d_1 < N^6} \left[\frac{N^{3\sigma + \epsilon - 1}}{r^{3-\epsilon} (d_1d_2)^{-2}} + \frac{N^{2\sigma + \epsilon - 1/2}}{r^{2-\epsilon}d_1^{1- \epsilon}d_2^{1-\epsilon}}   \right] \ll N^{3\sigma -4  + \epsilon} + N^{2\sigma - 5/2 + \epsilon}.
    \end{equation}
    This is $\ONE$ if $\supp (\hphi) \subset (-5/4, 5/4)$. Thus the only terms left are those with $m_1d_2 = m_2d_1$. But since $(d_1, d_2) = 1$, we must have that $m_1 = m d_1$ and $m_2 = m d_2$ for some $m \ge 1$. 

    Applying this to \eqref{eq: PkkN b divides N} gives

    \begin{align}\label{eq: PkkN last Q}
        \mathcal{P}(k_1, k_2, N) &= \frac{4\pi^2i^{k_1 + k_2}}{\varphi(N)^2} \sum_{m \le Y} \frac{1}{m^2} \sum_{\substack{d_1, d_2 \le Y/m }} \frac{1}{d_1^2d_2^2} \sum_{\substack{(r, N) = 1 \\ r < \min(N^6/d_1, N^6/d_2)}} \frac{Q^*(m^2d_1^2, rd_1N, m^2d_2^2, rd_2N)}{r^2} \nn \\
        &\hs{1} + \ONE.
    \end{align}
    Proposition \ref{lem: Qstar fail} gives that
    \begin{align}\label{eq: Qstar fail applied}
        &Q^*(m^2d_1^2, rd_1N, m^2d_2^2, rd_2N)  \nn \\
        &\hs{1} = \frac{ 4\psi(m^2 d_1^2 d_2^2, m^2 d_1^2 d_2^2, Nr)}{\varphi(d_1d_2rN)} R(m^2 d_1^2, d_1) R(m^2 d_2^2, d_2) \mu(d_1 d_2) \chi_0^{d_1d_2}(r) I(rd_1, rd_2, m d_1, md_2, N) \nn \\
        &\hs{2}+   O\left(m^6 d_1^3 d_2^3   N^{2\sigma - 1/2 + \epsilon} (d_1d_2r)^\epsilon \right).
    \end{align}
    By the properties of the Ramanujan sum we have that $\psi(ax, ay, z) = \psi(x,y,z)$ if $(a, z) = 1$. Since $(N,r) = 1$, $(m^2d_1^2d_2^2, N) = 1$ and $(d_1^2d_2^2, r) = 1$, we have that
    \begin{equation}
        \psi(m^2 d_1^2 d_2^2, m^2 d_1^2 d_2^2, Nr) = \psi(1, 1, N) \psi(m^2, m^2, r).
    \end{equation}
    We also have that $R(m^2 d_1^2, d_1) = \varphi(d_1)$ and $R(m^2 d_2^2, d_2) = \varphi(d_2)$. Lastly, from comparing the definitions \eqref{eq: I def} and \eqref{eq: IrN def} we have that
    \begin{align}
        I(rd_1, rd_2, m d_1, md_2, N) &= I(m,r, N).
    \end{align}
    Applying these identities to \eqref{eq: Qstar fail applied} gives
    \begin{align}\label{eq: Qstar fail final}
        Q^*(m^2d_1^2, rd_1N, m^2d_2^2, rd_2N) &= \frac{ 4\psi(1, 1, N) \psi(m^2, m^2, r)}{\varphi(rN)} \mu(d_1 d_2) \chi_0^{d_1d_2}(r) I(m, r, N) \nn \\
        &\hs{1}+   O\left(m^6 d_1^3 d_2^3   N^{2\sigma - 1/2 + \epsilon} (d_1d_2r)^\epsilon \right).
    \end{align}
    Applying this to \eqref{eq: PkkN last Q} gives
    \begin{align}\label{eq:P post Q}
        \mathcal{P}(k_1, k_2, N) &= \frac{16\pi^2i^{k_1 + k_2}\psi(1,1,N)}{\varphi(N)^3} \sum_{m \le Y} \frac{1}{m^2} \sum_{\substack{d_1, d_2 \le Y/m }} \frac{\mu(d_1d_2)}{d_1^2d_2^2} \sum_{\substack{(r, d_1d_2N) = 1 \\ r < \min(N^6/d_1, N^6/d_2)}} \frac{\psi(m^2, m^2, r)}{r^2 \varphi(r)} I(m, r, N)\nn \\
        &\hs{1} + O\left(N^{2\sigma - 5/2 + \epsilon} + N^{-\epsilon}\right).
    \end{align}
    If $\sigma < 5/4$, the error term vanishes in the limit as $N \to \infty$.
    To complete the proof of the lemma, we extend the sum to be over all $r$ with $(r, d_1d_2) = 1$ using $J_{k-1}(x) \ll x$, which gives
    \begin{equation}\label{eq: b re extend bound}
        I(m, r,N) \ll m^2 r^{-2}N^{3\sigma - 2}.
    \end{equation}
    Extending the sum over $r$ in \eqref{eq:P post Q} to include terms with $N|r$ introduces an error term of size
    \begin{equation}
        N\inv \sum_{m \le Y} \frac{1}{m^2} \sum_{d_1, d_2 \le Y/m} \frac{1}{d_1^2 d_2^2} \sum_{r'} \frac{\psi(m^2, m^2, r'N)}{(r'N)^2\varphi(r'N)} m^2(r'N)^{-2} N^{3\sigma -2} \ll N^{3\sigma -6 + \epsilon}
    \end{equation}
    which is $\ONE$ if $\sigma < 2$. Lastly, the terms with $r \ge \min(N^6/d_1, N^6/d_2)$ can be bounded by
    \begin{equation}
        N\inv \sum_{m \le Y} \frac{1}{m^2} \sum_{d_1, d_2 \le Y/m} \frac{1}{d_1^2 d_2^2} \sum_{r \ge N^{6- \epsilon}} \frac{\psi(m^2, m^2, r)}{r^2\varphi(r)} m^2 r^{-2} N^{3\sigma -2} \ll N^{3\sigma -15 + \epsilon}.
    \end{equation}
    This is $\ONE$ if $\sigma < 5$.
\end{proof}

\subsection{Evaluating the integral}\label{sec: IrN eval}
Let $k_1, k_2$ be arbitrary.
Unfolding the Fourier transform gives
\begin{equation}
    I(m, r, N) = \int_0^\infty J_{k_1 -1 }\left( \frac{4 \pi m y}{rN} \right) J_{k_2 -1 }\left( \frac{4 \pi m y}{rN} \right)  \infint \phi(x) y^{-4\pi i x /\log R }  dx \frac{dy }{\log R}.
\end{equation}
After doing a change of variables $x \to x \log R$, we want to interchange the integrals. However, the integral does not converge absolutely, so we introduce a parameter $\epsilon$, which gives
\begin{align}
    I(m, r, N) &= \lim_{\epsilon \to 0} \infint \phi(x \log R) \int_0^\infty J_{k_1 -1 }\left( \frac{4 \pi m y}{rN} \right) J_{k_2 -1 }\left( \frac{4 \pi m y}{rN} \right) y^{-\epsilon -4\pi i x} dy dx.
    \end{align}
The above integral is absolutely convergent for any $\epsilon >0$ due to the rapid decay of $\phi$ and the bound $J_{k-1}(x) \ll x^{-1/2}$ from Lemma \ref{lem:Bessel}. This allows us to apply Fubini's theorem and swap the order of integration. Set
\begin{equation}\label{eq: double Bessel Mellin transform}
    H(\nu, \mu, s) \clq \int_0^\infty J_\nu(x) J_\mu(x) x^{-s} dx
\end{equation}
which is essentially a Mellin transform. Setting $u = 4\pi m y/rN$ gives
    \begin{align}
    I(m, r,N) &= \lim_{\epsilon \to 0} \frac{rN}{4\pi m} \infint \phi(x \log R) \left( \frac{4\pi m }{rN}\right)^{\epsilon + 4\pi ix} H(k_1 - 1, k_2 - 1, \epsilon + 4\pi i x) dx.
\end{align}
Setting $s = \epsilon + 4\pi ix$, we reinterpret this as a contour integral:
\begin{equation}
    I(m, r,N) = \lim_{\epsilon \to 0} \frac{rN}{16\pi^2 i m} \int_{\Re(s) = \epsilon} \phi \left( \frac{(s -\epsilon) \log R}{4\pi i} \right) \left( \frac{4\pi m }{rN}\right)^{s} H(k_1 - 1, k_2 - 1, s) ds.
\end{equation}
Section 6.8 (33) of \cite{bateman1954tables} (which is given in (6.8) of \cite{kowalski2002rankin}) gives:
\begin{equation}\label{eq: H Bessel def}
    H(\nu, \mu, s) = 2^{-s} \frac{\Gamma(s) \Gamma(\frac{\nu +\mu  + 1-s}{2} )}{\Gamma(\frac{\nu -\mu  + 1+s }{2} ) \Gamma(\frac{\mu -\nu  +1+s}{2} ) \Gamma(\frac{\nu +\mu  + 1+s}{2} )},\quad 0 < \Re(s) < \nu + \mu + 1.
\end{equation}
Using the identity (essentially Euler's reflection formula)
\begin{equation}
    \Gamma(1/2 + s) \Gamma(1/2 - s) = \frac{\pi}{\cos \pi s}
\end{equation}
gives
\begin{equation}
    H(\nu, \mu, s) = 2^{-s}  \frac{\cos (\pi \frac{s-\nu+\mu}{2})  \Gamma(s) \Gamma(\frac{\nu +\mu+1 -s}{2} ) \Gamma(\frac{\nu -\mu + 1 -s }{2})}{ \pi  \Gamma(\frac{\nu +\mu +1+s }{2} ) \Gamma(\frac{\nu -\mu +1+s }{2} )  }.
\end{equation}
For our case where $\nu = k_1 - 1$ and $\mu = k_2 - 1$ with $k_1, k_2 \ge 2$ even, we have that
\begin{equation}\label{eq: H expansion}
    H(k_1 - 1, k_2 - 1, s) = 2^{-s} i^{k_1 + k_2} \frac{\cos (\pi \frac{s}{2}) \Gamma(s) \Gamma(\frac{k_1 + k_2 - 1 -s}{2} ) \Gamma(\frac{k_1 - k_2 + 1 -s }{2})}{\hfill\pi  \Gamma(\frac{k_1 + k_2 - 1+s }{2} ) \Gamma(\frac{k_1 - k_2 +1+s }{2} )  }, \quad 0 < \Re(s) < 3.
\end{equation}
We want to interchange the integral with the sum over $r$ in \eqref{eq: P post cancellation}. In order to make everything absolutely convergent, we shift the contour to the line $\Re(s) = 2 $ and rearrange, giving
\begin{align}\label{eq: contour shift}
    &\sum_{\substack{(r, d_1d_2) = 1 }} \frac{\psi(m^2, m^2, r)}{r^2 \varphi(r)} I(m, r, N)\nn \\
    &\hs{1}= \lim_{\epsilon \to 0} \frac{N}{16\pi^2 i m} \int_{\Re(s) = 2 } \phi \left( \frac{(s -\epsilon) \log R}{4\pi i} \right) \left( \frac{4\pi m }{N}\right)^{s} \chi(s) H(k_1 - 1, k_2 - 1, s) ds
\end{align}
where
\begin{equation}
    \chi(s) = \sum_{(r, d_1d_2) = 1} \frac{\psi(m^2,m^2,r)}{r\varphi(r)} r^{-s}
\end{equation}
is a Dirichlet series absolutely convergent when $\Re(s) > 1$.

Now, by Lemma \ref{lem: psi lem}, we have that
\begin{align}\label{eq: chi expansion}
    \chi(s) &= \prod_{p \nmid md_1d_2} \left[\frac{- 1}{p \varphi(p) p^{s}} + \frac{1}{1 - p^{-s}} \right] \prod_{\substack{p | m \\ p \nmid d_1d_2}} \left[ 1 + \frac{1}{p^{s+1}}\right] \nn \\
    &= \zeta(s) \zeta_{md_1d_2}(s)\inv \alpha_{md_1d_2}(s) \beta_{m/(m, (d_1d_2)^\infty)}(s+1)
\end{align}
where
\begin{align}
    \zeta_d(s) &\clq \prod_{p|d} \frac{1}{1-p^{-s}}\\
    \alpha_d(s) &\clq \prod_{p\nmid d} \left[1 - \frac{1- p^{-s}}{p\varphi(p) p^s}\right] \\
    \beta_d(s) &\clq \prod_{\substack{p | d}} \left[ 1 + p^{-s} \right].
\end{align}
By the properties of infinite products (see \cite{stein2010complex}, for example), $\alpha_d(s)$ converges absolutely when
\begin{equation}
    \sum_{p \nmid d} \left| \frac{1 - p^{-s}}{p \varphi(p) p^s} \right| < \infty,
\end{equation}
which is satisfied when $\Re(s) > -1/2$. Now, by the functional equation for the Riemann zeta function we have that
\begin{equation}
    \zeta(1-s) = 2^{1-s} \pi^{-s} \cos\left(\pi \frac{s}{2} \right) \Gamma(s) \zeta(s).
\end{equation}
Thus applying \eqref{eq: H expansion} and \eqref{eq: chi expansion} to \eqref{eq: contour shift} gives
\begin{align}\label{eq: zeta functional}
     \sum_{\substack{(r, d_1d_2) = 1 }}& \frac{\psi(m^2, m^2, r)}{r^2 \varphi(r)} I(m, r, N) =\lim_{\epsilon \to 0} \frac{Ni^{k_1 +k_2}}{32\pi^3im} \int_{\Re(s) = 2 } \phi \left( \frac{(s -\epsilon) \log R}{4\pi i} \right) \left( \frac{4\pi^2m }{N}\right)^{s}  \nn \\
     & \times  \zeta_{md_1d_2}(s)\inv \alpha_{md_1d_2}(s) \beta_{m/(m, (d_1d_2)^\infty)}(s + 1) \zeta(1-s) \frac{ \Gamma(\frac{k_1 + k_2 - 1 -s}{2} ) \Gamma(\frac{k_1 - k_2 + 1 -s }{2})}{  \Gamma(\frac{k_1 + k_2 - 1+s }{2} ) \Gamma(\frac{k_1 - k_2 +1+s }{2} )  }  ds.
\end{align}
We want to shift the contour back to the line $\Re(s) = \epsilon$. If $k_1 \ne k_2$, there are no poles. 
If $k_1 = k_2$, there is a pole at $s = 1$ coming from the term $\Gamma(\frac{k_1 - k_2 +1-s }{2} )$ with residue
\begin{equation}
     \phi \left( \frac{(1 -\epsilon) \log R}{4\pi i} \right) \left( \frac{4\pi^2m }{N}\right)^1 \zeta_{md_1d_2}(1)\inv \alpha_{md_1d_2}(1) \beta_{m/(m, (d_1d_2)^\infty)}(2) \zeta(0) \frac{ \Gamma(k_1 - 1 ) }{  \Gamma(k_1 ) \Gamma(1 )  } \cdot (-2).
\end{equation}
Taking $\epsilon \to 0$, using the functional equation for the Gamma function and $\zeta(0) = -1/2$, the above equals
\begin{equation}\label{eq: pole contrib}
    \frac{4\pi^2 m }{N(k_1 - 1)} \phi \left( \frac{ \log R}{4\pi i} \right)   \zeta_{md_1d_2}(1)\inv \alpha_{md_1d_2}(1) \beta_{m/(m, (d_1d_2)^\infty)}(2).
\end{equation}

We finish treating the contribution from the pole in Section~\ref{sec:pole}. For the rest of the section, assume that $k_1 \ne k_2$ so there is no pole term. Our analysis closely follows Section 7 of \cite{ILS}. Because the test function $\phi$ is Schwartz, the integrand decays rapidly on the line $\Re(s) = \epsilon$ (the other terms in the integrand can be bounded by polynomials). In particular, we have that for any $B > 0$ that
\begin{equation}
    \phi \left( \frac{(s -\epsilon) \log R}{4\pi i} \right) \ll (\Im(s)\log R)^{-B}, \qquad \Re(s) = \epsilon.
\end{equation}
Because of this, the integral is $O(\log \inv R)$ outside the region $\Im(s) \ll \log ^{-1/2} R$. Since we are also taking $\epsilon \to 0$, we can use the Laurent expansion for our functions around $s = 0$. We have that
\begin{align}
    \zeta(1-s) &= -s\inv + O(1) \\
    \zeta_d(s)\inv &= \delta(1, d) + O(s \log d) \\
    \alpha_d(s) &= 1 + O(s) \\
    \beta_d(s + 1) &= \frac{\nu(d)}{d} + O(s)
\end{align}
where $\nu(d)$ is defined in \eqref{eq:nu def}. By Section 7 of \cite{ILS} (see the middle of page 100) we have that
\begin{equation}
    \Gamma\left(\frac{k-s}{2} \right) = \Gamma\left(\frac{k+s}{2} \right) \left(\frac{k}{2}\right)^{-s} \left[1 + O\left( \frac{s}{k}\right) \right].
\end{equation}
Applying these to \eqref{eq: zeta functional} gives 
\begin{align}\label{eq: sum with zeta functional}
     &\sum_{\substack{(r, d_1d_2) = 1 }} \frac{\psi(m^2, m^2, r)}{r^2 \varphi(r)} I(m, r, N) \nn \\
     &\hs{1}= - \lim_{\epsilon \to 0} \frac{Ni^{k_1 +k_2}}{32\pi^3i } \delta(1,md_1d_2) \int_{\substack{\Re(s) =  \epsilon \\ \Im(s) \ll \log^{-1/2} R}} \phi \left( \frac{(s -\epsilon) \log R}{4\pi i} \right) A^{-s/2} \frac{ds}{s} \nn \\
     &\hs{2}  + O\left( N\log(md_1d_2) \log\inv R \right)
\end{align}
where
\begin{equation}
    A \clq \frac{(k_1 + k_2 - 1)^2(k_1 - k_2 + 1)^2N^2}{256\pi^4 }.
\end{equation}
Note that the main term is only nonzero when $m= 1$. Applying this to \eqref{eq: P post cancellation} and re-extending the integral to the entire line $\Re(s) = \epsilon$ using the decay of $\phi$ gives
\begin{equation}
    \mathcal{P}(k_1, k_2, N) = -\frac{N\psi(1,1,N)}{\varphi(N)^3 2\pi i} \lim_{\epsilon \to 0} \int_{\Re(s) =  \epsilon} \phi \left( \frac{(s -\epsilon) \log R}{4\pi i} \right) A^{-s/2} \frac{ds}{s} + \OlogR.
\end{equation}
Now, setting $s = \epsilon + 4\pi ix$ and doing a change of variables $\epsilon \to 2\epsilon$ gives
\begin{equation}\label{eq: P penultimate}
    \mathcal{P}(k_1, k_2, N) = -\frac{N\psi(1,1,N)}{\varphi(N)^3 } \lim_{\epsilon \to 0} A^{-\epsilon} \infint \phi \left( x \log R \right) A^{ - 2\pi ix} \frac{dx}{\epsilon + 2\pi ix} + \OlogR.
\end{equation}
By Section 7 of \cite{ILS}, we have that
\begin{equation}
    \lim_{\epsilon \to 0} A^{-\epsilon} \infint \phi \left( x \log R \right) A^{ - 2\pi ix} \frac{dx}{\epsilon + 2\pi ix} = -\infint \phi(x) \frac{\sin(2\pi x)}{2\pi x} dx + \frac{1}{2}\phi(0) + \OlogR
\end{equation}
and by Lemma \ref{lem: psi lem} we have that
\begin{equation}
    \frac{N\psi(1,1,N)}{\varphi(N)^3} = 1 +O(N\inv).
\end{equation}
Applying this to \eqref{eq: P penultimate} and using that $\log \inv R \asymp \log \inv N$, we finally have
\begin{equation}
    \mathcal{P}(k_1, k_2, N) = \infint \phi(x) \frac{\sin(2\pi x)}{2\pi x} dx - \frac{1}{2}\phi(0) + O\left( \log \inv N \right)
\end{equation}
as desired. \qed

\section{Handling the poles}\label{sec:pole}
In this section, we complete the proof of Theorem \ref{thm: extendo} by accounting for the case where $k_1 = k_2$. For the remainder of the section, set $N = N_1 = N_2$ and $k = k_1 = k_2$ with $k$ fixed.

We begin by bounding the complementary sum $\Delta_{k, N}^\infty$ appearing in \eqref{eq: sum over f} and \eqref{eq: sum over g}. In Section \ref{sec:applying}, we bounded this sum using Lemmas \ref{lem: ILS 2.12} and \ref{lem: double complementary sum} and setting $Y = N^{\epsilon}$. The complementary sums are larger in the case where there are poles, so we set $Y = N^\alpha$ with $\alpha = 1/14$ so that we can use the $Y$ decay and extend the support slightly past $(-1, 1)$. We need to choose $\alpha$ small enough so that the error terms vanish in the limit, but large enough so that the complementary $m_1, m_2$ sums vanish. This amounts to taking $\sigma$ as large as possible under the constraints
\begin{align}
    \sigma - \alpha/2 - 1 &\le 0 \\
    2\sigma + 6\alpha - 5/2 &\le 0.
\end{align}
The optimal solution is $\alpha = 1/14$, $\sigma = 29/28$. The first equation comes from \eqref{eq: PkN 4.1} and the second from \eqref{eq: 7 qstar error term}.

We repeat the analysis in Section \ref{sec:extendo} in the case where $k_1 = k_2$. Our analysis is mostly the same as in that section, so we omit some details. The main difference is that the size of $m_1, m_2$ is no longer negligible, as now we have that $m_1, m_2 \le N^\alpha$ instead of $m_1, m_2 \le N^{\epsilon}$. Our main result is the following.
\begin{prop}\label{prop: 7 god}
Let $\supp(\hphi) \subset (-29/28, 29/28)$ and set
    \begin{align}\label{eq: PkN def}
    \mathcal{P}(k, N) &\clq \frac{1}{|H(k, N, k, N)|} \sum_{f,g} S(f \ot g; \phi).
\end{align}
We have that
\begin{equation}
     \mathcal{P}(k, N) = \infint \phi(x) \frac{\sin(2\pi x)}{2\pi x} dx - \frac{1}{2}\phi(0) + \frac{2}{|H_k^*(N)|}\phi\left(\frac{\log R}{4\pi i} \right) + O\left( \log \inv N \right).
\end{equation}
\end{prop}
Combining Proposition \ref{prop: 7 god} with Proposition \ref{prop: explicit formula} completes the proof of Theorem \ref{thm: extendo} in the case where $k_1 = k_2$ after comparing with \eqref{eq: W def}. 
\begin{proof}
First we apply Proposition \ref{prop: Petersson final} with $Y = N^\alpha$, which gives
\begin{align}\label{eq: PkN 4.1}
    \mathcal{P}(k, N) &= \frac{4\pi^2}{\varphi(N)^2} \sum_{m_1, m_2 \le N^\alpha} \frac{1}{m_1m_2} \sum_{b_1, b_2  \ge 1} \frac{1}{b_1b_2} Q^*(m_1^2,b_1N,m_2^2, b_2N) + O\left(N^{- \alpha/2 + \epsilon} + N^{\sigma - \alpha/2 - 1 + \epsilon} \right).
\end{align}
The error term is $\ONE$ if $\sigma < 29/28$.
Similarly to Lemma \ref{lemma: small b}, we restrict the sum over $b_1, b_2$ to be up to $N$, which introduces an error term of size $N^{3 \sigma +2\alpha -7/2 +\epsilon}$. This is $\ONE$ if $\sigma < 29/28$. Next, we remove all subterms except where $m_1 = md_1$ and $m_2 = md_2$ as in Lemma \ref{lem: full restriction}. Using \eqref{eq: Qstar bound not equal}, we find that this introduces an error term of 
\begin{equation}\label{eq: 7 qstar error term}
    N^{3\sigma +2\alpha - 4+ \epsilon} + N^{2\sigma +6\alpha - 5/2 + \epsilon}.
\end{equation}
This is $\ONE$ if $\sigma < 29/28.$ Applying Proposition \ref{lem: Qstar fail} gives
\begin{align}\label{eq: PkN post cancellation}
        \mathcal{P}(k, N) &= \frac{16\pi^2\psi(1,1,N)}{\varphi(N)^3} \sum_{m \le N^\alpha}  \frac{1}{m^2} \sum_{\substack{d_1, d_2 \le N^{\alpha}/m }} \frac{\mu(d_1d_2)}{d_1^2d_2^2} \sum_{\substack{r <\min(N/d_1, N/d_2) \\ (r, d_1d_2) = 1 }} \frac{\psi(m^2, m^2, r)}{r^2 \varphi(r)} I(m, r, N)\nn \\
        &\hs{1} + \ONE.
\end{align}
Using the bound \eqref{eq: b re extend bound}, we extend the sum to be over all $r$ with $(r, d_1d_2) = 1$. This introduces an error term of size
\begin{equation}
    N\inv \sum_{m \le N^\alpha} 1 \sum_{d_1, d_2 \le N^{\alpha}/m} \frac{1}{d_1^2d_2^2} \sum_{r \ge N^{1-\alpha}} \frac{1}{r^3} N^{3\sigma - 2} \ll N^{3\sigma + 3\alpha - 5}.
\end{equation}
This is $\ONE$ if $\sigma < 29/28$. This gives
\begin{align}
    \mathcal{P}(k, N) &= \frac{16\pi^2\psi(1,1,N)}{\varphi(N)^3} \sum_{m \le N^\alpha}  \frac{1}{m^2} \sum_{\substack{d_1, d_2 \le N^{\alpha}/m }} \frac{\mu(d_1d_2)}{d_1^2d_2^2} \sum_{\substack{(r, d_1d_2) = 1 }} \frac{\psi(m^2, m^2, r)}{r^2 \varphi(r)} I(m, r, N)\nn \\
        &\hs{1} + \ONE.
\end{align}
Our analysis of the sum over $r$ is the same as in Section \ref{sec: IrN eval} except for the contribution from the pole in \eqref{eq: zeta functional}. By \eqref{eq: pole contrib} and the Cauchy residue theorem, the contribution of the pole to $\mathcal{P}(k, N)$ is 
\begin{align}
    &2\pi i\times \frac{16\pi^2\psi(1,1,N)}{\varphi(N)^3} \sum_{m \le N^\alpha} \frac{1}{m^2} \sum_{\substack{d_1, d_2 \le N^{\alpha}/m }} \frac{\mu(d_1d_2)}{d_1^2d_2^2} \frac{N}{32\pi^3 i m} \nn \\
    &\hs{1} \times  \frac{4\pi^2 m }{N(k - 1)} \phi \left( \frac{ \log R}{4\pi i} \right)   \zeta_{md_1d_2}(1)\inv \alpha_{md_1d_2}(1) \beta_{m/(m, (d_1d_2)^\infty)}(2).
\end{align}
Using that $N\psi(1,1,N)/\varphi(N)^3 = 1 + O(N\inv)$ and \eqref{eq: H size}, we can simplify this as
\begin{align}\label{eq: get rid of N}
    &\frac{2}{|H_k^*(N)|} \phi \left( \frac{ \log R}{4\pi i} \right) \frac{\pi^2}{6} \sum_{m \le N^\alpha} \frac{1}{m^2} \sum_{\substack{d_1, d_2 \le N^{\alpha}/m }} \frac{\mu(d_1d_2)}{d_1^2d_2^2}   \zeta_{md_1d_2}(1)\inv \alpha_{md_1d_2}(1) \beta_{m/(m, (d_1d_2)^\infty)}(2) \nn \\
    & \hs{1} + O\left(N^{\sigma - 2} \right).
\end{align}
The error term in \eqref{eq: get rid of N} is $\ONE$ if $\sigma < 29/28$ and comes from \eqref{eq: pole bound} and the bound  
\begin{equation}\label{eq: zeta garbage bound}
    \zeta_{md_1d_2}(1)\inv \alpha_{md_1d_2}(1) \beta_{m/(m, (d_1d_2)^\infty)}(2) \ll 1.
\end{equation}
We utilize \eqref{eq: pole bound} and \eqref{eq: zeta garbage bound} to extend the sum over $d_1, d_2$ to be over all positive integers. This introduces an error term of size
\begin{equation}\label{eq: reextend d1d2}
    N^{\sigma -1} \sum_{m \le N^\alpha} \frac{1}{m^2} \sum_{d_1 \ge 1} \frac{1}{d_1^2} \sum_{d_2 > N^{\alpha}/m} \frac{1}{d_2^2}  \ll N^{\sigma - \alpha - 1 + \epsilon}.
\end{equation}
Likewise, we extend the sum over $m$ to be over all positive integers, which introduces an error term of the same size. These error terms are $\ONE$ if $\supp \hphi \subset (-29/28, 29/28)$. 
Thus we have that the contribution from the pole term is
\begin{equation}\label{eq: pole contrib 7}
    \frac{2}{|H_k^*(N)|} \phi \left( \frac{ \log R}{4\pi i} \right) \frac{\pi^2}{6}  \sum_{\substack{d_1, d_2 \ge 1 }} \frac{\mu(d_1d_2)}{d_1^2d_2^2} \sum_{m \ge 1} \frac{1}{m^2}   \zeta_{md_1d_2}(1)\inv \alpha_{md_1d_2}(1) \beta_{m/(m, (d_1d_2)^\infty)}(2) + O\left( N^{-\epsilon} \right).
\end{equation}
We complete the proof by calculating the sums over $d_1, d_2$ and $m$.
\begin{lemma}\label{lem: the last one}
    Set
    \begin{equation}
        S = \sum_{\substack{d_1, d_2 \ge 1 }} \frac{\mu(d_1d_2)}{d_1^2d_2^2} \sum_{m \ge 1} \frac{1}{m^2}   \zeta_{md_1d_2}(1)\inv \alpha_{md_1d_2}(1) \beta_{m/(m, (d_1d_2)^\infty)}(2).
    \end{equation}
    We have that
    \begin{equation}
        S = \frac{6}{\pi^2}.
    \end{equation}
\end{lemma}
\begin{proof}
We have that
\begin{align}
   \zeta_{md_1d_2}(1)\inv \alpha_{md_1d_2}(1) &= \zeta(3)\inv \prod_{p |d_1d_2} \frac{p^3 - p^2}{p^3 - 1} \prod_{\substack{q |m \\ q \nmid d_1d_2}} \frac{q^3 - q^2}{q^3 - 1} \\
    \beta_{m/(m, (d_1d_2)^\infty)}(2) &= \prod_{\substack{p |m \\ p \nmid d_1d_2}} \frac{p^2+1}{p^2}
\end{align}
where $q$ is prime. This gives 
\begin{align}
    \sum_{m \ge 1} \frac{1}{m^2}   \zeta_{md_1d_2}(1)\inv \alpha_{md_1d_2}(1) \beta_{m/(m, (d_1d_2)^\infty)}(2) &= \zeta(3)\inv \prod_{p |d_1d_2} \frac{p^3 - p^2}{p^3 - 1} \sum_{m \ge 1} \frac{1}{m^2} \prod_{\substack{q |m \\ q \nmid d_1d_2}} \frac{q^3 - q^2}{q^3 - 1} \frac{q^2+1}{q^2}.
\end{align}
Factoring the sum into an Euler product, we have that the above equals
\begin{equation}
    \zeta(3)\inv \prod_{p |d_1d_2} \frac{p^3 - p^2}{p^3 - 1} \frac{1}{1 - p^{-2}} f(p)\inv \prod_{q} f(q)
\end{equation}
where
\begin{equation}
    f(p) = 1 + \frac{p^3 - p^2}{p^3 - 1}\frac{p^2 + 1}{p^2} \frac{1}{p^2 - 1}.
\end{equation}
Thus we have 
\begin{equation}
    S = \zeta(3) \inv \prod_{q} f(q) \sum_{\substack{d_1, d_2 \ge 1 }} \frac{\mu(d_1d_2)}{d_1^2d_2^2} \prod_{p | d_1d_2} \frac{p^3 - p^2}{p^3 - 1} \frac{1}{1 - p^{-2}} f(p)\inv.
\end{equation}
We want to count the number of times that a fixed value of $d_1d_2$ appears in the above sum. Since $d_1d_2$ can be assumed to be squarefree, if $d_1d_2$ has $a$ prime factors, then it appears $2^a$ times. Thus we have that
\begin{align}
    S &= \zeta(3) \inv \prod_{q} f(q) \prod_p \left[ 1 - \frac{2}{p^2} \frac{p^3 - p^2}{p^3 - 1} \frac{1}{1 - p^{-2}} f(p)\inv \right] \nn \\
    &= \zeta(3)\inv \prod_p \left[ f(p) - \frac{2}{p^2} \frac{p^3 - p^2}{p^3 - 1} \frac{1}{1 - p^{-2}}  \right] .
\end{align}
Now, we have that
\begin{equation}
    f(p) - \frac{2}{p^2} \frac{p^3 - p^2}{p^3 - 1} \frac{1}{1 - p^{-2}} = \frac{1-p^{-2}}{1-p^{-3}}
\end{equation}
so that 
\begin{align}
    S &= \zeta(3)\inv \prod_p \frac{1-p^{-2}}{1-p^{-3}} \nn \\
    &= \zeta(3) \inv \zeta(3) \zeta(2)\inv \\
    &= \zeta(2)\inv = \frac{6}{\pi^2}.
\end{align}
\end{proof}
Applying Lemma \ref{lem: the last one} to \eqref{eq: pole contrib 7} gives that the contribution from the pole is 
\begin{equation}
    \frac{2}{|H_k^*(N)|} \phi \left( \frac{ \log R}{4\pi i} \right) + O\left(N^{-\epsilon} \right).
\end{equation}
Combining this with the analysis of the integral from Section \ref{sec: IrN eval} gives the proposition.
\end{proof}

\newpage
\bibliographystyle{alpha}
\bibliography{bib}

@article{ILS,
  title={Low lying zeros of families of {$ L $-functions}},
  author={Iwaniec, H. and Luo, W. and Sarnak, P.},
  journal={Publications Math{\'e}matiques de l'IH{\'E}S},
  volume={91},
  pages={55--131},
  year={2000}
}

@article{moeglin1989poles,
  title={Le spectre r{\'e}siduel de {$GL(n)$}},
  author={Moeglin, C. and Waldspurger, J. L.},
  journal={Ann. Sci. {\'E}cole Norm. Sup.(4)},
  volume={22},
  pages={605--674},
  year={1989}
}

@article{FI,
  title={Low-lying zeros of dihedral {$L$}-functions},
  author={Fouvry, E. and Iwaniec, H.},
  journal={Duke Mathematical Journal},
  volume={116},
  number={2},
  pages={189--217},
  year={2003},
  publisher={Duke University Press}
}

@article{DM,
  title={The effect of convolving families of {$L$-functions} on the underlying group symmetries},
  author={Due{\~n}ez, E. and Miller, S. J.},
  journal={Proceedings of the London Mathematical Society},
  volume={99},
  number={3},
  pages={787--820},
  year={2009},
  publisher={Wiley Online Library}
}

@article{kowalski2002rankin,
  title={{Rankin-Selberg $L$-functions} in the level aspect},
  author={Kowalski, E. and Michel, P. and VanderKam, J.},
  journal={Duke Mathematical Journal},
  volume={114},
  number={1},
  pages={123--191},
  year={2002},
  publisher={Duke University Press}
}

@article{li1979series,
  title={{$L$-series of Rankin} type and their functional equations},
  author={Li, W.},
  journal={Mathematische Annalen},
  volume={244},
  number={2},
  pages={135--166},
  year={1979},
  publisher={Springer}
}

@article{petersson1932,
  title={{\"U}ber die {Entwicklungskoeffizienten der automorphen Formen}},
  author={Petersson, H.},
  journal={Acta mathematica},
  volume={58},
  pages={169--215},
  year={1932},
  publisher={Institut Mittag-Leffler}
}

@book{iwaniec1997topics,
  title={Topics in classical automorphic forms},
  author={Iwaniec, H.},
  series={Graduate Studies in Mathematics},
  volume={17},
  year={1997},
       address={Providence, RI},
  publisher={AMS}
}

@book{diamond2005modular,
  author={Diamond, F. and Shurman, J.},
  title={A First Course in Modular Forms},
  year={2005},
  address = {New York, NY},
  series = {Graduate Texts in Mathematics},
  volume = {228},
  publisher={Springer}
}

@book{ono2004web,
  title={The Web of Modularity: Arithmetic of the Coefficients of Modular Forms and $ q $-series},
  author={Ono, K.},
  year={2004},
number={102},
series={Conference Board of the Mathematical Sciences},
address={Providence, RI},
  publisher={AMS}
}

@book{stein2010complex,
  title={Complex analysis},
  author={Stein, E. M. and Shakarchi, R.},
  series = {Princeton Lectures in Analysis},
  volume={2},
  year={2003},
  publisher={Princeton University Press}
}

@book{bateman1954tables,
  title={Tables of integral transforms},
  author={Erd{\'e}lyi, A. and Magnus, W. and Oberhettinger, F. and Tricomi, F. G.},
  volume={1},
  year={1954},
  publisher={McGraw-Hill book company}
}

@book{IK,
  title={Analytic number theory},
  author={Iwaniec, H. and Kowalski, E.},
  volume={53},
series={AMS Colloquium Publications},
   publisher={AMS},
     address={Providence, RI},
  year={2004},
}

@book{MV,
  title={Multiplicative number theory I: Classical theory},
  author={Montgomery, H. L. and Vaughan, R. C.},
  year={2007},
  publisher={Cambridge University Press}
}

@article{HM, 
    title={Low-lying zeros of {$L$}-functions with orthogonal symmetry}, 
    volume={136}, 
    DOI={10.1215/s0012-7094-07-13614-7}, 
    number={1}, 
    pages={115-172},
    journal={Duke Mathematical Journal}, 
    author={Hughes, C. P. and Miller, S. J.}, 
    year={2007}
}

@article{RS, 
    title={Zeros of principal {$L$}-functions and random matrix theory}, 
    journal={Duke Mathematical Journal}, 
    author={Z. Rudnick and P. Sarnak}, 
    volume = {81},
    number = {2},
    year={1996},
    pages={269-322}
}

@book{KS1,
      author={N.~M. Katz and P. Sarnak},
       title={Random matrices, {F}robenius eigenvalues, and monodromy},
      series={AMS Colloquium Publications},
   publisher={AMS},
     address={Providence, RI},
        year={1999},
      volume={45},
        ISBN={0-8218-1017-0},
}

@article{KS2,
      author={N.~M. Katz and P. Sarnak},
       title={Zeroes of zeta functions and symmetry},
        year={1999},
        ISSN={0273-0979},
     journal={Bull. Amer. Math. Soc. (N.S.)},
      volume={36},
      number={1},
       pages={1-26},
         url={http://dx.doi.org/10.1090/S0273-0979-99-00766-1},
}

@article{selberg1940bemerkungen,
  title={{Bemerkungen {\"u}ber eine Dirichletsche Reihe, die mit der Theorie der Modulformen nahe verbunden ist}},
  author={Selberg, A.},
  journal={Arch. Math. Nat.},
  volume={43},
  year={1940},
pages={47-50}
}

@article{rankin1939contributions,
  title={Contributions to the theory of {Ramanujan's function $\tau(n)$} and similar arithmetical functions: {II}. the order of the {Fourier} coefficients of integral modular forms},
  author={Rankin, R. A.},
  journal={Mathematical Proceedings of the Cambridge Philosophical Society},
  volume={35},
  number={3},
  pages={357--372},
  year={1939},
  organization={Cambridge University Press}
}

@article{keating2000random,
  title={Random matrix theory and {$L$-functions} at $s= 1/2$},
  author={Keating, J. P. and Snaith, N. C.},
  journal={Communications in Mathematical Physics},
  volume={214},
  pages={91--100},
  year={2000},
  publisher={Springer}
}

@article{rad2023conditional,
      title={Conditional lower bounds on the distribution of central values in families of {$L$-functions}}, 
      author={M. Radziwi{\l}{\l} and K. Soundararajan},
      journal={Acta Arithmetica},
  volume={214},
  pages={481--497},
  year={2024}
}

@article{drappeau2023one,
  title={One-level density estimates for {Dirichlet} {$L$-functions} with extended support},
  author={Drappeau, S. and Pratt, K. and Radziwi{\l}{\l}, M.},
  journal={Algebra \& Number Theory},
  volume={17},
  number={4},
  pages={805--830},
  year={2023},
  publisher={Mathematical Sciences Publishers}
}

@article{devin2022low,
  title={Low-lying zeros in families of holomorphic cusp forms: the weight aspect},
  author={Devin, L. and Fiorilli, D. and S{\"o}dergren, A.},
  journal={The Quarterly Journal of Mathematics},
  volume={73},
  number={4},
  pages={1403--1426},
  year={2022},
  publisher={Oxford University Press UK}
}

@phdthesis{yang2009,
    title    = {Distribution problems associated to zeta functions and invariant theory},
    school   = {Princeton University},
    author   = {Yang, A.},
    year     = {2009}
}

@inproceedings{sarnak2016families,
  title={Families of {$L$-functions} and their symmetry},
  author={Sarnak, P. and Shin, S. W. and Templier, N.},
  booktitle={Families of automorphic forms and the trace formula},
  pages={531--578},
  series={Simons Symposia},
  year={2016},
  organization={Springer}
}

@article{knightly2019weighted,
  title={Weighted Distribution of {Low-lying Zeros} of {$GL(2)$ $L$}-functions},
  author={Knightly, A. and Reno, C.},
  journal={Canadian Journal of Mathematics},
  volume={71},
  number={1},
  pages={153--182},
  year={2019},
  publisher={Canadian Mathematical Society}
}

@article{waxman2021lower,
  title={Lower order terms for the one-level density of a symplectic family of {Hecke $L$-functions}},
  author={Waxman, E.},
  journal={Journal of Number Theory},
  volume={221},
  pages={447--483},
  year={2021},
  publisher={Elsevier}
}

@inproceedings{Mon,
  title={The pair correlation of zeros of the zeta function},
  author={Montgomery, H. L.},
  booktitle={Proc. Symp. Pure Math},
  volume={24},
  pages={181--193},
  year={1973}
}

@inproceedings{Alpoge2015,
author="L. Alpoge
and N. Amersi
and G. Iyer
and O. Lazarev
and S. J. Miller
and L. Zhang",
title="Maass Waveforms and Low-Lying Zeros",
bookTitle="Analytic Number Theory: In Honor of {Helmut Maier's} 60th Birthday",
year="2015",
publisher="Springer International Publishing",
pages="19--55",
isbn="978-3-319-22240-0",
doi="10.1007/978-3-319-22240-0_2",
url="https://doi.org/10.1007/978-3-319-22240-0_2"
}

@article{carneiro2022hilbert,
  title={Hilbert spaces and low-lying zeros of {$L$-functions}},
  author={Carneiro, E. and Chirre, A. and Milinovich, M. B.},
  journal={Advances in Mathematics},
  volume={410},
  pages={108748},
  year={2022},
  publisher={Elsevier}
}

@article{alpoge2015low,
  title={Low-lying zeros of {Maass} form {$L$}-functions},
  author={Alpoge, L. and Miller, S. J.},
  journal={International Mathematics Research Notices},
  volume={2015},
  number={10},
  pages={2678--2701},
  year={2015},
  publisher={Oxford University Press}
}

@article{barrett2017one,
  title={One-level density for holomorphic cusp forms of arbitrary level},
  author={Barrett, O. and Burkhardt, P. and DeWitt, J. and Dorward, R. and Miller, S. J.},
  journal={Research in Number Theory},
  volume={3},
  pages={1--21},
  year={2017},
  publisher={Springer}
}

@article{Guloglu2005,
    author = {Güloğlu, A. M.},
    title = "{Low-lying zeroes of symmetric power {$L$}-functions}",
    journal = {International Mathematics Research Notices},
    volume = {2005},
    number = {9},
    pages = {517-550},
    year = {2005},
    month = {01},
    issn = {1073-7928},
    doi = {10.1155/IMRN.2005.517},
    url = {https://doi.org/10.1155/IMRN.2005.517},
    eprint = {https://academic.oup.com/imrn/article-pdf/2005/9/517/1824383/2005-9-517.pdf},
}

@article{Entin2012,
author = {Entin, A. and Roditty-Gershon, E. and Rudnick, Z.},
year = {2013},
pages = {1230-1261},
title = {Low-lying Zeros of Quadratic {Dirichlet} {$L$-Functions}, Hyper-elliptic Curves and Random Matrix Theory},
volume = {23},
journal = {Geometric and Functional Analysis},
doi = {10.1007/s00039-013-0241-8}
}

@article{Shin2012,
  title={{Sato–Tate} theorem for families and low-lying zeros of automorphic {$L$}-functions},
  author={S. W. Shin and N. Templier},
  journal={Inventiones mathematicae},
  year={2016},
  issue = {1},
  volume={203},
  pages={1-177}
}

@article{duenez2006low,
  title={The low lying zeros of a {GL (4) and a GL (6)} family of {$L$-functions}},
  author={Due{\~n}ez, E. and Miller, S. J.},
  journal={Compositio Mathematica},
  volume={142},
  number={6},
  pages={1403--1425},
  year={2006},
  publisher={London Mathematical Society}
}

@article{FiMi,
  title={Surpassing the ratios conjecture in the 1-level density of {Dirichlet $L$-functions}},
  author={Fiorilli, D. and Miller, S. J.},
  journal={Algebra \& Number Theory},
  volume={9},
  number={1},
  pages={13--52},
  year={2015},
  publisher={Mathematical Sciences Publishers}
}

@article{MP,
title = {Low-lying zeros of number field {$L$}-functions},
journal = {Journal of Number Theory},
volume = {132},
number = {12},
pages = {2866-2891},
year = {2012},
issn = {0022-314X},
doi = {https://doi.org/10.1016/j.jnt.2012.05.034},
url = {https://www.sciencedirect.com/science/article/pii/S0022314X12001928},
author = {S. J. Miller and R. Peckner}
}

@article{OS1,
  title={Small zeros of quadratic {$L$-functions}},
  author={{\"O}zl{\"u}k, A. E. and Snyder, C.},
  journal={Bulletin of the Australian Mathematical Society},
  volume={47},
  number={2},
  pages={307--319},
  year={1993},
  publisher={Cambridge University Press}
}

@article{OS2,
  title={On the distribution of the nontrivial zeros of quadratic {$L$-functions} close to the real axis},
  author={{\"O}zl{\"u}k, A. E. and Snyder, C.},
  journal={Acta Arithmetica},
  volume={91},
  number={3},
  pages={209--228},
  year={1999}
}

@article{Ricotta2007,
author = {Ricotta, G. and Royer, E.},
year = {2011},
pages = {969-1028},
title = {Statistics for low-lying zeros of symmetric power {$L$}-functions in the level aspect},
volume = {23},
journal = {Forum Mathematicum},
doi = {10.1515/FORM.2011.035}
}

@article{Yo,
author = {Young, M.},
year = {2006},
title = {Low-lying zeros of families of elliptic curves},
volume={19},
number={1},
pages={205--250},
journal = {Journal of the American Mathematical Society},
doi = {10.1090/S0894-0347-05-00503-5}
}

@article{cohen2022extending,
  title={Extending support for the centered moments of the low lying zeroes of cuspidal newforms},
  author={Cohen, P. and Dell, J. and Gonz{\'a}lez, O. E. and Iyer, G. and Khunger, S. and Kwan, C.-H. and Miller, S. J. and Shashkov, A. and Reina, A. S. and Sprunger, C. and Triantafillou, N. and Truong, N. and Van Peski, R. and Willis, S. and Yang, Y.},
  journal={arXiv preprint arXiv:2208.02625},
  year={2022}
}

@article{Gao, title={$n$-Level Density of the Low-lying Zeros of Quadratic {Dirichlet} {$L$}-Functions}, volume={2014}, DOI={10.1093/imrn/rns261}, number={6}, journal={International Mathematics Research Notices}, author={Gao, P.}, year={2014}, pages={1699–1728}}

@article{Mil2,
  title={One-and two-level densities for rational families of elliptic curves: evidence for the underlying group symmetries},
  author={Miller, S. J.},
  journal={Compositio Mathematica},
  volume={140},
  number={4},
  pages={952--992},
  year={2004},
  publisher={London Mathematical Society}
}

@article{Ru, 
title={Low-lying zeros of {$L$}-functions and random matrix theory}, 
journal={Duke Mathematical Journal}, 
year={2001}, 
pages={147-181}, 
volume={109},
number={1},
author={Rubinstein, M.}}

@article{Ro, 
     author = {E. Royer},
     title = {Petits z\'eros de fonctions {$L$} de formes modulaires},
     journal = {Acta Arithmetica},
     volume = {99},
    number = {2},
     year = {2001},
     pages = {147-172},
     zbl = {0984.11024},
     language = {fra},
     url = {http://dml.mathdoc.fr/item/bwmeta1.element.bwnjournal-article-doi-10_4064-aa99-2-3}
}

@article{boldyriew2023determining,
  title={Determining optimal test functions for 2-level densities},
  author={Bo{\l}dyriew, E. and Chen, F. and Devlin, C. and Miller, S. J. and Zhao, J.},
  journal={Research in Number Theory},
  volume={9},
  number={2},
  pages={32},
  year={2023},
  publisher={Springer}
}

@article{dutta2022bounding,
  title={Bounding the Order of Vanishing of Cuspidal Newforms via the nth Centered Moments},
  author={Dutta, S. and Miller, S. J.},
  journal={arXiv preprint arXiv:2211.04945},
  year={2022}
}

\end{document}